\documentclass[a4paper,11pt,reqno]{amsart}
\usepackage[utf8]{inputenc}
\usepackage{pdfsync}
\usepackage{verbatim}
\usepackage[linktocpage]{hyperref}
\usepackage{amssymb, paralist, xspace, graphicx, url, amscd, euscript, mathrsfs,stmaryrd,epic,eepic,color}
\usepackage{enumitem}
\usepackage[all]{xy}
\usepackage{amsthm}

\usepackage{tikz}
\usepackage{tikz-cd}
\usetikzlibrary{calc}
\usetikzlibrary{arrows}
\usetikzlibrary{matrix}
\usepackage{lipsum}
\SelectTips{cm}{}
\usepackage[margin=1in]{geometry}

\numberwithin{equation}{section}

1
\numberwithin{subsection}{section}

\allowdisplaybreaks[1]

\newtheorem*{namedtheorem}{\theoremname}
\newcommand{\theoremname}{testing}

\newtheorem{theorem}{Theorem}[subsection]
\newtheorem{proposition}[subsection]{Proposition}
\newtheorem{conjecture}[subsection]{Conjecture}
\newtheorem{corollary}[subsection]{Corollary}

\newtheorem{lemma}[subsection]{Lemma}
\theoremstyle{definition}
\newtheorem{definition}[subsection]{Definition}

\newtheorem{example}[subsection]{Example}
\newtheorem{remark}[subsection]{Remark}

\theoremstyle{remark}

\newcommand\cA{\mathcal{A}}
\newcommand\cB{\mathcal{B}}
\newcommand\cC{\mathcal{C}}

\newcommand\cF{\mathcal{F}}

\newcommand\cH{\mathcal{H}}

\newcommand\cL{\mathcal{L}}
\newcommand\cM{\mathcal{M}}
\newcommand\cN{\mathcal{N}}
\newcommand\cO{\mathcal{O}}
\newcommand\cP{\mathcal{P}}

\newcommand\cT{\mathcal{T}}

\newcommand\cX{\mathcal{X}}
\newcommand\cY{\mathcal{Y}}
\newcommand\cZ{\mathcal{Z}}

\newcommand\uE{\underline{E}}

\renewcommand\AA{\mathbb{A}}

\newcommand\CC{\mathbb{C}}

\newcommand\DD{\mathbb{D}}

\newcommand\FF{\mathbb{F}}

\newcommand\HH{\mathbb{H}}

\newcommand\KK{\mathbf{K}}

\newcommand\PP{\mathbb{P}}
\newcommand\QQ{\mathbb{Q}}
\newcommand\RR{\mathbb{R}}

\newcommand\ZZ{\mathbb{Z}}

\newcommand\bF{\mathbf{F}}

\newcommand\rD{\mathrm{D}}

\newcommand\rH{\mathrm{H}}

\newcommand\rO{\mathrm{O}}

\newcommand\rS{\mathrm{S}}

\newcommand\fX{\mathfrak{X}}

\newcommand{\DB}{\mathrm{DB}}

\newcommand\Sh{{\rm Sh}}
\newcommand\Spec{{\rm Spec~}}
\newcommand\Pic{{\rm Pic}}

\newcommand{\SO}{{\rm SO}}

\newcommand{\Hom}{{\rm Hom}}
\newcommand{\bv}{{\bf v}}

\newcommand{\Mp}{{\rm Mp}}
\newcommand{\Sp}{{\rm Sp}}

\newcommand{\SL}{{\rm SL}}

\newcommand{\CH}{{\rm CH}}
\newcommand{\cl}{{\rm cl}}

\newcommand{\cosk}{{\rm cosk}}
\newcommand{\et}{{\rm \acute{e}t}}
\newcommand{\Et}{{\rm \acute{E}t}}
\newcommand{\an}{{\rm an}}
\newcommand{\Zar}{{\rm Zar}}
\newcommand\scF{\mathscr{F}}
\newcommand\scS{\mathscr{S}}
\newcommand\scX{\mathscr{X}}
\newcommand{\q}{/\!\!/}

\theoremstyle{plain}
\theoremstyle{definition}

\begin{document}
	
\title[DB cohomology of the universal K3 surface]{Deligne-Beilinson cohomology of the universal K3 surface}

\author{Zhiyuan Li}
\address{Shanghai Center for Mathematical Sciences \\ Fudan University\\
220 Handan Road, Shanghai, 200433 China\\
	}
\email{zhiyuan\_li@fudan.edu.cn}

\author{Xun Zhang}
\address{Math Department \\ Fudan University\\
220 Handan Road, Shanghai, 200433 China\\
}
\email{18110180041@fudan.edu.cn}

\begin{abstract}
O'Grady's generalized Franchetta conjecture (GFC) is concerned with codimension 2 algebraic cycles on universal polarized K3 surfaces. In \cite{BL17},  this conjecture has been studied in the Betti cohomology groups.  Following a suggestion of Voisin, we investigate this problem in the  Deligne-Beilinson (DB) cohomology groups. In this paper,  we develop the theory of Deligne-Beilinson cohomology groups on separated (smooth) Deligne-Mumford stacks. Using the automorphic cohomology group and Noether-Lefschetz theory, we compute the 4-th DB-cohomology group of universal oriented polarized K3 surfaces with at worst an $A_1$-singularity and show that GFC for such family holds in DB-cohomology. In particular, this confirms O'Grady's original conjecture in DB cohomology. 
\end{abstract}


\maketitle
	
\section{Introduction}
Zero-cycles on K3 surfaces have been studied by Beauville and Voisin in \cite{BV04}. They have shown that if $S$ is a K3 surface, then there is a canonical zero-cycle $c_S$ on $S$ that satisfies the following two properties:
 \begin{enumerate}
 	\item The intersection of two divisor classes on $S$ lies in $\mathbb{Z} c_S \subset \CH_0 (S)$.
 	\item The second Chern class $c_2 (T_S )$ equals $24 c_S \in \CH_0 (S)$.
 \end{enumerate}
This canonical cycle is called Beauville-Voisin class in $\CH_0(S)$. A natural question is to investigate the codimension 2 cycles on the universal K3 surfaces. More precisely,  let $\scF_{g}^\circ$ be the moduli space of primitively polarized complex K3 surfaces of genus $g>2$ with only trivial automorphism groups, which carries a universal family 
$$\pi : \scS_{g}^\circ\rightarrow \scF_g^\circ.$$  
O'Grady \cite{O'Grady} has conjectured that the $0$-th Chow group with rational coefficients  of the generic fiber of $\pi$ is  spanned by the Beauville-Voisin class. This is referred to as the generalized Franchetta conjecture which can be viewed as a higher dimensional analogue to the Franchetta conjecture on $\cM_g$. 
\begin{conjecture}[Generalized Franchetta Conjecture for K3 surfaces] \label{conj1}
Let $\cT_\pi$ be the relative tangent bundle of $\pi:\scS_g^\circ\to \scF_g^\circ$. For any $\alpha\in \CH^2(\scS^\circ_{g})$,  there exists $m\in \QQ$ such that $\alpha-m c_2(\cT_\pi)$ is supported on a proper subset of $\scF_g^\circ$. 
\end{conjecture}
This conjecture has been confirmed when $g\leq 12$ and $g=15$ ({\it cf}.~ \cite{PSY, FL2021}) by using the explicit geometric models. It remains open in general.  Recently, Beauville has  shown in \cite{Beauville2021} that there exists for every $g$ a hypersurface in $\scF_g^\circ$ such that GFC holds for the corresponding family. We also refer the readers to \cite{FU2019, FU2021,La2019} for the generalization to hyper-K\"ahler varieties. In \cite{BL17}, the  authors have verified GFC  in cohomology groups.   In this paper, we give more evidences to support this conjecture by considering the cycle classes of the universal K3 surface in the Deligne-Beilinson cohomology group.  Recall that the Deligne-Beilinson cohomology $\rH^{\bullet}_{\DB}(X_\CC,\ZZ(p))$ of a quasi-projective complex algebraic variety $X_\CC$ is defined as the hypercohomology of the Deligne-Beilinson complex (cf.~\cite{EV88,Voi07}). Denote by 
\begin{equation}\label{Decycle}
\cl_{\DB}: \CH^k(X_{\CC})\rightarrow \rH^{2k}_{\DB}(X_\CC,\QQ(k)). 
\end{equation}
the DB-cycle class map. Compared with the ordinary cohomology group, the DB cohomology group is  closer to the Chow groups because the cycle class map $\cl:\CH^k(X_\CC)\to \rH^{2k}(X_\CC, \QQ(k))$ factors through $\cl_{\DB}$.  The main result of this article is 

\begin{theorem}\label{mainthm} 
For any $\alpha\in \CH^2(\scS_{g}^\circ)$, there exists a rational number $m\in \QQ$ such that the class 
\begin{equation}\label{prop-supp}
    \cl_{\DB}(\alpha-m c_2(\cT_\pi)) \in \rH^4_{\DB}(\scS_{g}^\circ, \QQ(2))
\end{equation}
is supported on a proper closed subset of $\scF_{g}^\circ$, i.e.  the restriction of $ \cl_{\DB}(\alpha-m c_2(\cT_\pi))$ is zero on $\pi^{-1}(V)$ for some open subset $V\subseteq \scF_g^\circ$. 
\end{theorem}
We shall mention that the generalized Franchetta conjecture over $\overline{\QQ} $ can be deduced from  the Bloch-Beilinson conjecture.   When $X$ is defined over a number field, Beilinson  conjectured that  the rational Chow ring $\CH^\ast(X_{\bar{\QQ}})$  injects into the DB cohomology \cite[Conj. 2.4.2.1]{Bei84} via the cycle class map $\cl_{\DB}$.
Note that the universal family $\scS^\circ_{g}$ can be defined over a number field, so whenever Beilinson's conjecture  holds  either for $\scS^\circ_g$ or for K3 surfaces, one can  immediately obtain the generalized Franchetta conjecture over $\overline{\QQ}$.  

Let us conclude this introduction by explaining  the ideas of the proof. As cohomological generalized Franchetta conjecture is known, it suffices to calculate the kernel of the map $$\rH^4_{\DB}(\scS_g^\circ, \QQ(2))\to \rH^4(\scS_g^\circ,\QQ(2)),$$ which lies in $\rH^3(\scS_g^\circ,\CC)/F^2\rH^3(\scS_g^\circ,\CC)$. 
In \cite{BL17}, the low-degree cohomology groups (with coefficients) of the moduli stack $\scF_g$ of quasi-polarized K3 surfaces and its special moduli substacks have been completely computed. But
the cohomology groups of $\scF_g^\circ$ seem rather difficult to compute due to the complexity of the boundary of $\scF_g^\circ$ in $\scF_g$.   To overcome this difficulty,  we have to work with the separated substack of some covering of $\scS$.   In our situation, we consider the separated locus of the moduli stack of quasi-polarized oriented K3 surfaces (which is a double covering of $\scF_g^\circ$) and our proof of Theorem \ref{mainthm} relies on a dedicated calculation on the cohomology groups of this open substack.

\subsection*{Organization of the paper}
In Section \ref{Sec2}  and Section \ref{Sec3} , we review the basic theory of Shimura varieties of orthogonal type and moduli stacks of quasi-polarized K3 surfaces.   In Section \ref{Sec4}, we introduce the DB cohomology on separated Deligne-Mumford stacks and construct the DB-cycle class map for smooth ones. As expected,   we show that it is parallel to the traditional DB cohomology theory for varieties.  In Section \ref{Sec5}, we consider the moduli stack of quasi-polarized oriented K3 surfaces and study the separated locus of this stack. 
Section \ref{Sec6} is devoted to proving Theorem \ref{mainthm}. We perform the NL-number computation on families of  unigonal and hyperelliptic K3 surfaces. Combined with the results in \cite{BL17}, this allows us to show that the  fourth DB cohomology group of universal family injects into the Betti cohomology. The assertion of Theorem \ref{mainthm} follows immediately. 

\subsection*{Acknowledgement}
The authors are grateful to Claire Voisin for drawing our attention to this problem and Nicolas Bergeron for lots of helpful discussions and useful comments. We would like to thank Xiping Zhang for the discussion on NL-number computations. The authors were supported by NSFC grants (No.~12121001 and No.~12171090).

\section{Preliminary: Shimura variety and special cycle}  \label{Sec2}
In this section, we recall some basic results on Shimura varieties of orthogonal type which will be used in the proof. 

\subsection{Shimura variety of orthogonal type}Let $V$ be a non-degenerate  quadratic even lattice  over $\ZZ$ of signature $(2,n)$ and let $G=\SO(V)$ be the corresponding special orthogonal group. The group $G(\RR)$ of real points of $G$ is isomorphic to $\SO (2,n)$. Fix a maximal compact subgroup of $K_\RR \subset G(\RR)$ and let $\rD_V$ be the connected component of $G(\RR)/K_\RR$.  It is a type IV Hermitian symmetric domain of complex dimension $n$. 
 
Let $\hat{\ZZ}$ be the profinite completion of $\ZZ$ and $\AA_f=\hat{\ZZ}\otimes \QQ$ be the ring of finite adeles.   Let $\KK\subseteq G(\AA_f)$  be the discriminant kernel, i.e. the largest subgroup of $G(\hat{\ZZ})$ that acts trivially on the discriminant group $V^\vee/V$. To any compact open subgroup $K\subseteq \KK$ of $G(\AA_f)$ it corresponds a quotient space
$$\Sh_K(V) = G(\QQ)\backslash \widehat{\rD}_V \times G (\AA_f) /K.$$
We shall simply write $\Sh(V)$ when $K = \KK$. 

If $K$ is neat, e.g. when $K$ is sufficiently small,  then $\Sh_K(V)$ is a smooth quasi-projective variety.  In general, $\Sh_K(V)$ is not connected. The connected component corresponding to the coset of $\rD_V \times \{\mathrm{id} \}$  is an arithmetic quotient $$X_{\Gamma}:=\Gamma \backslash \rD_V$$ where $\Gamma=G(\QQ)\cap K$.   
 If $V$ contains a hyperbolic plane, then one has $$\Sh(V)=\Gamma_V\backslash \rD_V,$$where $\Gamma_V=\KK\cap G(\ZZ)$ is the stable special orthogonal group. 

\subsection{Special cycles on $X_\Gamma$}  
Given a vector $\bv\in (V^\vee)^r\subseteq V^r(\QQ)$, we let $U=U(\bv)$ be the $\QQ$-subspace of $V(\QQ)$ spanned by the components of $\bv$. Let $\rD_{\bv}\subset \rD_V$ be the subset consisting of $2$-planes which lie in $U^\perp$. 
The codimension $r$ natural cycle $Z(\bv)$ on $X_\Gamma$ is defined to be the image of 
\begin{equation}\label{special-manifold}
\Gamma_{\bv}\backslash \rD_\bv\rightarrow \Gamma\backslash \rD_V.
\end{equation}
where $\Gamma_{\bv}$ is the stabilizer of $U$ in $\Gamma$. For any $\beta\in \mathrm{Sym}_{r\times r} (\QQ)$  of rank $t(\beta)$, we set 
$$\Omega_\beta=\{\bv\in V^r \; | \;  \frac{1}{2} (\bv,\bv)=\beta,\dim U(\bv)=t(\beta)\}.$$
which comes with a natural $\Gamma$-action. 
For each function $\varphi\in \Hom((V^\vee/V)^r, \CC)$,  Kudla associates a {\it special cycle} $$Z(\beta,\varphi;\Gamma)\in \CH^r(X_\Gamma)$$ 
as  some linear combinations of $\lambda^{r-t(\beta)} \cdot Z(\bv)$ with $\bv\in \Omega_\beta$ (cf. \cite{Ku97}). Here $\lambda$ is the Hodge line bundle on $X_\Gamma$. They are compatible with natural pullback maps $\pi:X_{\Gamma'}\to X_\Gamma$ for $\Gamma'\subseteq \Gamma$, i.e. $$\pi^\ast(Z(\beta,\varphi;\Gamma))=Z(\beta,\varphi;\Gamma').$$  Hence we may write  $Z(\beta,\varphi)=Z(\beta,\varphi;\Gamma)$ for simplicity.   Let us  denote by $Z(\beta)$ the map  $$\varphi \mapsto  Z(\beta, \varphi ) \in  \Hom((V^\vee/V)^r, \CC)^\vee \otimes  \CH^r(X_\Gamma)_\CC .$$ 

Let $\Mp_{2r}(\ZZ)$ be the metaplectic double covering of $\Sp_{2r}(\ZZ)$ and $\rho_r$ be the Weil representation of $\Mp_{2r}(\ZZ)$ acting on $\Hom((V^\vee/V)^r, \CC)$. A celebrated theorem for special cycles on $X_\Gamma$ is 
\begin{theorem}[Kudla's modularity conjecture]\label{Modularity}
The generating series 
 \begin{equation}
   \Theta(\tau)= \sum\limits_{\beta\geq0} Z(\beta)  e^{2\pi i \mathrm{tr}(\beta\tau)},~ \tau\in \HH_r
 \end{equation}
 with coefficients in $\Hom((V^\vee/V)^r, \CC)^\vee \otimes_\CC \CH^r(X_\Gamma)_\CC$ is a  Siegel modular form of type $\rho_r$, weight $\frac{n+2}{2}$  and genus $r$ in with values in $\CH^r(X_\Gamma)_\CC$, i.e. 
$$\Theta(A\tau)=\det(c\tau +d)^{\frac{n+2}{2}}\rho_r (A, \pm \sqrt{\det(c\tau+d)})\Theta(\tau),$$
for all $A=\left(\begin{array}{ccc} a & b \\ c & d \end{array}\right)\in \Sp_{2r}(\ZZ)$. 
\end{theorem}

\begin{example}\label{Modularity-unimodular}
If $V$ is unimodular, then $V^\vee/V$ is trivial and the special cycle $Z(\beta)$ can be written as the quotient by $\Gamma$ of the union of hyperplanes
\begin{equation} \label{specialcycle}
Z(\beta)= \Gamma\backslash \sum\limits_{\bv\in \Omega_\beta}  U(\bv)^\perp.
\end{equation}
which is a linear combination of $Z(\bv)$ for $\bv\in \Gamma\backslash \Omega_\beta$.  
In this case, Theorem \ref{Modularity} reads as the generating series $\Theta(\tau)$ is a classical Siegel modular form of weight $\frac{n+2}{2}$ and genus $r$.
\end{example}

Another important result concerning the special cycles is the following
\begin{theorem}[\cite{BL17}] \label{BL-vanishing}
 Let $E$ be a  $\CC$-representation of $\rO(V_\RR)$ and  let $\underline{E}$ be the associated  local system on $X_\Gamma$.  
 Suppose $n\geq 8$, then we have  
\begin{enumerate}
		\item  $\rH^1(X_\Gamma,\uE)=\rH^3(X_\Gamma,\uE)=0$.
		\item  $\rH^{2i}(X_\Gamma),\uE)$ is spanned by  special cycles  with coefficients in $\uE$ if $i=1$ or $2$. 
	\end{enumerate}	
\end{theorem}

\section{Moduli spaces of quasi-polarized K3 surfaces}\label{Sec3} 

In this section, we give an overview of quasi-polarized K3 surfaces and discuss some Noether-Lefschetz cycles.
\subsection{Moduli space of (quasi-)polarized K3 surfaces} A {\it (quasi-)polarized K3 surface of genus $g\geq 2$}  is a pair $(S,L)$,  where $S$ is a K3 surface and  $L\in \Pic(S)$ is an ample (resp.~nef) line bundle with $L^2=2g-2$.  The middle cohomology $\rH^2(S,\ZZ)$ equipped with the bilinear intersection form $q(-,-)$ is an even unimodular lattice which is isomorphic to the {\it K3 lattice}
$$\Lambda=U^{\oplus 3}\oplus E_8(-1)^{\oplus 2}.$$
where $U$ is the hyperbolic lattice of rank two and $E_8$ is the positive definite lattice associated to the Lie algebra of the same name. Let $\{e_1,f_1\}$ be a standard basis of the first hyperbolic lattice.  The primitive part $\rH^2_{\rm prim}(S,\ZZ):=c_1(L)^\perp$ has signature $(2,19)$ and it is isomorphic to the lattice
$$\Lambda_{g}:=\left<e_1-(g-1)f\right>\oplus U^{\oplus 2}\oplus E_8(-1)^{\oplus 2}, $$ 
by identifying $c_1(L)$ with the vector $e_1+(g-1)f_1$. 
The Hodge structures on $\rH^2_{\rm prim}(S,\ZZ)$ are parametrized by the period domain $\rD_{\Lambda_g}$ associated to $\Lambda_g$.  
 
Let  $\scF_g$ be the moduli stack of primitively quasi-polarized complex K3 surfaces of genus $g$ and $\scF_{g}^p \subseteq \scF_g$ the open substack of polarized K3 surfaces of genus $g.$  The following result is well-known. 

\begin{theorem}(cf.~\cite{Huy12})
The moduli stack $\scF_g$ is a smooth Deligne-Mumford stack of finite type and it is coarsely represented by a quasi-projective variety $\cF_g$ and there is an isomorphism 
$$\cF_g\to \widetilde{\rO}(\Lambda_g)\backslash \rD_{\Lambda_g}.$$
Moreover, the open substack $\scF_g^p$ is separated. 
\end{theorem}

More generally, one can consider the moduli space of lattice polarized K3 surfaces, which behaves similarly as $\scF_g$. 
\begin{definition}[Lattice polarized K3]    
Let $S$ be a projective K3 surface and $\Sigma$ a non-degenerate lattice.  We say that $S$ is a $\Sigma$-(quasi-)polarized K3 surface if  there exists a primitive embedding $\iota: \Sigma\to \Pic(S)$ such that the image of $\Sigma$ in $\Pic(S)$ under $\iota$ contains an ample (resp. nef and big) class.
\end{definition}
Let $\cF_{\Sigma}$ be the coarse moduli space of $\Sigma$-quasi-polarized K3 surfaces. In \cite{Do96}, Dolgachev proved that
\begin{theorem}(cf.~\cite[Remark 3.4]{Do96})\label{moduli_space_of_lattice_polarized_K3}
Fix a primitive embedding of $\Sigma$ into $\Lambda$ and consider $\Sigma$ as a sublattice of $\Lambda$ via this embedding. Then there is an isomorphism 
$$\cF_\Sigma\cong \Gamma_{\Sigma}\backslash \rD_{\Sigma^\perp}$$
where $\rD_{\Sigma^\perp}:= \{[\sigma]\in \rD_{\Lambda}|~q(\sigma, \Sigma)=0\}$ is the period domain of $\Sigma$-(quasi-)polarized K3 surfaces and $\Gamma_{\Sigma}:=\{\gamma\in \rO(\Lambda)|~\gamma|_{\Sigma}=id\}.$
\end{theorem}

\subsection{Example: Moduli of elliptic K3 surfaces}\label{Elliptic}
Let us consider the $U$-quasi-polarized K3 surfaces. If $S$ is $U$-quasi-polarized, we have $L_1,L_2\in\Pic(S)$ satisfying $L_i^2=0$ and $L_1L_2=1$. We may assume that $L_1$ is effective without fixed component. Then the map $$f:S\to |L_1|\cong \PP^1$$ is an elliptic fibration and the difference $L_1-L_2$ represents a section class of $f$.  Conversely, if $f:S\to \PP^1$ is an elliptic fibration with a section $O$, then the fiber class $F$ and $F+O$ gives a $U$-quasi-polarization.  We may also call  $S$ a {\it unigonal} K3 surface. Let $\cF_U$ be the coarse moduli space of $U$-quasi-polarized K3 surfaces. According to Theorem \ref{moduli_space_of_lattice_polarized_K3}, we have
$$\cF_U=\Gamma_{U} \backslash \rD_{U^\perp}.$$ 
Note that the orthogonal complement $U^\perp\cong U^{\oplus 2}\oplus E_8^{\oplus 2}$ is a unimodular lattice.

Let us follow \cite[III.2]{Mir} to describe the projective models of $U$-quasi-polarized K3 surfaces. Let $\pi:S\to \PP^1$ be the elliptic fibration and $C$ a section of $\pi$. Then $S$ has a unique Weierstrass model with at worst ADE singularities, i.e. we can think it as the closed subscheme of $\PP(\cO_{\PP^1}\oplus \cO_{\PP^1}(-4)\oplus \cO_{\PP^1}(-6))$  defined by the equation
$$y^2z=x^3+axz^2+bz^3$$ 
where $a\in \rH^0(\PP^1,\cO_{\PP^1}(8))$ and $b\in \rH^0(\PP^1,\cO_{\PP^1} (12))$ satisfying certain conditions (cf.~\cite[Theorem 2.1]{Mi81}). As shown in \cite{Mi81}, such equations are naturally parametrized by an open subset of the weighted projective space $\PP(2^{(9)},3^{(13)})$  equipped with a $\SL_2$-action.  Recently,  Odaka-Oshima \cite{OO21} have shown that there is actually an isomorphism     
\begin{equation}\label{OO-iso}
\PP(2^{(9)},3^{(13)}) \q \SL_2\cong \cF_U^\ast,
\end{equation}
between the GIT quotient and the Baily-Borel compactification of $\cF_U$.

There is another description of the projective model of $\pi:S\to \PP^1$: as
$$\pi_\ast \cO_S(2C)\cong \cO_{\PP^1}\oplus \cO_{\PP^1}(-4),$$
the natural surjection $\pi^\ast\pi_\ast \cO_S(2C) \to \cO_S(2C)$ defines a double covering $$S\to \bF_4=:\PP(\cO_{\PP^1}\oplus \cO_{\PP^1}(-4)).$$ Its Weierstrass model is a double covering of $\bF_4$ branched over the trisection $x^3+axz^2+bz^3=0$ and the section $z=0$.  Conversely, Let $A=\PP(\cO_{\PP^1}(-4))$ and $F$ a ruling of $\bF_4$. Choose $B\in |3A+12F|$ which is disjoint from $A$ and smooth, then the double covering of $\bF_4$ branched over $(A+B)$ is a unigonal K3 surface.

\subsection{Noether-Lefschetz loci on $\cF_g$}The Noether-Lefschetz (NL) locus $\cN \cL^1(\cF_g)\subset \cF_g$ is the locus parametrizing $K3$ surfaces in $\cF_{g}$ with Picard number greater than $1$. It is a union of countably many divisors defined as follows:  given $d,n\in\ZZ$, the {\it  NL-divisor}  $\cH_{d,n}^{g}\subset \cF_{g}$ is the locus of quasi-polarized K3 surfaces $(S,L)\in \cF_g$ whose Picard lattice $\Pic(S)$ contains a  rank two sublattice
\begin{equation}\label{lattice}\begin{array}{c|c|c} & L & \alpha \\\hline L & 2g-2 & d \\\hline \alpha & d & n\end{array}
\end{equation}
where  $\alpha\in \Pic(S)$. In general, $\cH_{d,n}^g$ is not irreducible and one can define the so called {\it irreducible NL-divisor} $\cP_{d,n}^g$ by requiring  the lattice \eqref{lattice} to be primitive. 
As divisors,  $\cH_{d,n}^g$ can be written as a linear combination of primitive NL divisors via the triangulated relation
\begin{equation}\label{triangular}
    \cH^{g}_{d,n}=\sum\limits_{(d_i,n_i) } \mu_{d_i,n_i} \cP^{g}_{d_i,n_i}
\end{equation}
where $(d_i,n_i)$ runs over $\ZZ^2$ modulo the relation $(d,n)\sim (d', n')$ if \begin{equation}\label{relationprim}
   d^2-2n(g-1)=(d')^2-2n'(g-1)\quad \hbox{and} \quad d\equiv d'\mod 2g-2;
  \end{equation}   the coefficient $\mu_{d_i,g_i}\in \{0,1,2\}$ is the number of integer solutions $(x,y)$ of the equations
$$ (d^2_i-2n_i(g-1))x^2=d^2-2n(g-1); ~(2g-2)y=d-x n_i.$$

Let $\Delta_{d,g}$ be the determinant of the matrix \eqref{lattice} and let $\varpi$ be the generator of the cyclic group $\Lambda_g^\vee/\Lambda_g$. According to \cite[Lemma 3]{MP13}, the projection $\overline{\alpha} $ of $\alpha$ to $\rH^2_{prim}(S,\ZZ)$ is a vector with norm 
$\frac{\Delta_{d,g}}{4g-4}$ and the image of $\overline{\alpha}$ in $ \Lambda^\vee_g/\Lambda_g$ is $d\varpi$, thus
the NL divisor $\cH^g_{d,n}$ can be identified as the special cycle of codimension one 
\begin{equation}
    \cZ(\frac{\Delta_{d,g}}{4g-4}, d):=\tilde{\rO} (\Lambda_g)\backslash \sum_{\substack{ v^2=\frac{\Delta_{d,g}}{4g-4},\\ v\equiv d \varpi \mod \Lambda_g }} v^\perp
\end{equation}
Note that $ \cZ(\frac{\Delta_{d,g}}{4g-4}, d)$ has multiplicity two if $2d\varpi=0$ in $\Lambda_g^\vee/\Lambda_g$ because $v$ and $-v$ will both occur in taking the sum. 

Similarly, one can define the higher codimensional Noether-Lefschetz loci $\cN\cL^m(\cF_g)$ parametrizing quasi-polarized K3 surfaces in $\cF_g$ with Picard number greater than $m$.  To specify the irreducible components of $\cN\cL^m(\cF_g)$, we  define the NL-cycles of codimension $m$ on $\cF_g$ as below.  
\begin{definition}
Let $M$ be a symmetric matrix of signature $(1,m)$.  We define $$\cH_{M}^g\subseteq \cN\cL^m(\cF_g)$$ to be the locus of quasi-polarized K3 surfaces $(S,L)\in\cF_g$ for which there exist classes $$\alpha_1,\ldots, \alpha_m\in \Pic(S)$$ such that the gram matrix of ($L, \alpha_1, \ldots, \alpha_m)$ is $M$.
\end{definition}

Using the period map, the NL-cycle $\cH_M^g$ can be identified as linear combinations of special cycles in the following way:  assume $L\cdot \alpha_i=d_i\in \ZZ$. Let $\overline{\alpha}_i$ be the projection of $\alpha_i$ in $\rH^2_{\rm prim}(S,\ZZ)$ and let $\overline{M}$ be the gram matrix of $\overline{\alpha_i}$. Let $\Omega_{M}$  be the collection of vectors  $\bv=(v_1,\ldots,v_m)\in \Lambda^m_g$ such that 
\begin{itemize}
    \item the gram matrix of  $(v_1,\ldots, v_m)$ is $\overline{M}$,
\item $v_i\equiv d_i \varpi \mod \Lambda_g$.
\end{itemize}
Then $\cH_M^g$ can be written as the quotient by  $\widetilde{\rO}(\Lambda_g)$ of the sum of subdomains
\begin{equation}\label{eq：heegnercycle}
   \widetilde{\rO}(\Lambda_g)\backslash \sum\limits_{\substack{\{v_1,\ldots,v_m\} \\\mathrm{s.t.} (v_1,\ldots, v_m)\in \Omega_M}} U(v_1,\ldots, v_m)^\perp  
\end{equation}
where the sum is taken over the $m$-ary subsets of $\Lambda_g$ consisting of components of vectors in $\Omega_M$. 
One can easily see the irreducible components of $\cH^g_M$ are the natural cycles. A simple fact is that  $\cH_M^g$ is irreducible if and only if all the sublattices $\mathtt{M}\subseteq \Lambda_g$ containing $e_1+(g-1)f_1$, having gram matrix $M$ with respect to $(e_1+(g-1)f_1, v_1, \ldots, v_m)$ for some $(v_1,\ldots,v_m)\in \mathtt{M}^m$, are lying in the same $\widetilde{\rO}(\Lambda_g)$-orbit. In the rest of this section, we will give some examples of NL-cycles and compute the irreducible components.

\subsection{First example: locus of exceptional K3 surfaces}A typical example is the NL-locus of K3 surfaces in $\cF_g$ with $(-2)$-exceptional curves. To describe the irreducible components of this locus, we need  the following result which characterizes the $\widetilde{\rO}(V)$-orbits of primitive vectors for certain $V$. 
 
\begin{proposition}[Eichler's criterion]  \label{Eichler}
Suppose that $V$ is an even integral lattice and it contains two copies of hyperbolic lattice. Then two primitive vectors $x,y\in V$ are lying  in the same $\widetilde{\rO}(V)$-orbit if and only if  $x^2=y^2$ and $x^\ast=y^\ast \mod V$, where $x^\ast, y^\ast$ are the dual of $x$ and $y$ in $V^\vee$. 
\end{proposition}
\begin{proof}
See \cite[Lemma 3.5]{GHS10}.
\end{proof}

Then we have 
\begin{proposition}({\it cf}.~\cite[Proposition 2.11]{Debarre18})\label{nodalloci}The complement $\cF_g\backslash \cF_g^p$ is exactly the NL-divisor $\cH^g_{0,-2}$.
It is irreducible if $g \not\equiv 2\mod 4  $ and it is the union of $\cP_{0,-2}^g$ and $\cP_{g-1,\frac{g-2}{2}}^g$ if $g\equiv 2 \mod 4$.
\end{proposition}
\begin{proof}
The irreducible components of $\cH^g_{0,-2}$ correspond to the $\widetilde{\rO}(\Lambda_g)$-orbits of roots in $\Lambda_g$. According to Proposition \ref{Eichler}, any two roots $x$ and $y$ are lying in the same $\widetilde{\rO}(\Lambda_g)$-orbit if $x^\ast=y^\ast$ in $\Lambda_g^\vee/\Lambda_g$. As the discriminant group $$\Lambda_g^\vee/\Lambda_g\cong \ZZ/(2g-2)\ZZ$$ it is not difficult to show that $x^\ast$ is either $0$ or $g-1$ in $\Lambda_g^\vee/\Lambda_g$. The latter situation occurs only when $g\equiv 2\mod 4$. The assertion follows easily. 
\end{proof}

\subsection{Second example: binodal and cuspidal loci}Let us describe some Noether-Lefschetz loci of codimension $2$ on $\cF_g$ which plays an important role in this paper.  Let $\{t_1,\cdots, t_8\}$  be a standard basis of the $E_8$ lattice. Then we introduce two sublattices in $\Lambda$: let $\mathtt{A_{1,1}}$ be the sublattice  generated  by $e_1+(g-1)f_1$, $t_1$ and $t_3$ and let $\mathtt{A}_{2}$  be the  sublattice generated by $e_1+(g-1)f_1$, $t_1$ and $t_2$. The associated gram matrices under $q(-,-)$ are 
\begin{center}
$A_{1,1}=\left(\begin{array}{ccc} 2g-2 & 0 & 0 \\ 0 & -2 & 0 \\ 0 & 0 & -2\end{array}\right)$;  $ A_{2}=\left(\begin{array}{ccc} 2g-2 & 0 & 0 \\ 0 & -2 & 1 \\ 0 & 1 & -2\end{array}\right).$\end{center}
We let $\cH^g_{A_{1,1}}$ and $\cH^g_{A_2}$ be the associated NL-cycles of codimension $2$. Indeed, they parametrize quasi-polarized K3 surfaces in $\cF_g$ with at least two exceptional $(-2)$-curves. We may call $\cH_{A_{1,1}}^g$ the binodal locus and $\cH_{A_2}^g$ the cuspidal locus. This is because for  K3 surfaces in $\cH^g_{A_{1,1}}$ and $\cH^g_{A_2}$, after contracting the two exceptional $(-2)$-curves, one can obtain singular K3 surfaces with two isolated $A_1$-singularities or respectively, an isolated $A_2$-singularity.

\begin{proposition}\label{irred-comp}
$\cH_{A_2}^g$ is irreducible and $\cH_{A_{1,1}}^g$ is in one of the following situations:
\begin{enumerate}
    \item if $g\equiv 2\mod 4$, it is the union of two irreducible components $\cH_{A_{1,1}}^{g'}$ and $\cH_{A_{1,1}}^{g''}$, where $\cH_{A_{1,1}}^{g''}$ is lying in the intersection of two NL-divisors $\cP^g_{0,-2}$ and $\cH^g_{g-1,\frac{g-2}{2}}$. 
    \item if $g\equiv 3\mod 4$,  it is the union of two irreducible components $\cH_{A_{1,1}}^{g'}$ and $\cH_{A_{1,1}}^{g'''}$, where $\cH_{A_{1,1}}^{g'''}$ is lying in the intersection of two NL-divisors $\cP^g_{0,-2}$ and $\cH^g_{g-1,\frac{g-3}{2}}$.  
    \item  $\cH_{A_{1,1}}^g$ is irreducible otherwise.
\end{enumerate}

\end{proposition}
\begin{proof} Let us first consider the case $\cH^g_{A_{1,1}}$.  Identify $\widetilde{\rO}(\Lambda_g)$ as the subgroup of $\rO(\Lambda)$ fixing the vector $e_1+(g-1)f_1$, then we just need to classify the $\widetilde{\rO}(\Lambda_g)$-orbits of lattice $M\subset \Lambda$ satisfying $e_1+(g-1)f_1\in \mathtt{M}$ and $\mathtt{M}\cong \mathtt{A_{1,1}}$  via an isometry fixing the vector $e_1+(g-1)f_1$.  One should note that $\cP^g_{g-1,\frac{g-2}{2}}\cap \cP^g_{g-1,\frac{g-2}{2}}=\emptyset $ by the Hodge index theorem.    According to Proposition \ref{nodalloci},  $ \cH^g_{A_{1,1}}$ is contained in the primitive NL-locus $\cP^g_{0,-2}$.  Hence up to a $\widetilde{\rO}(\Lambda_g)$-action, we can assume that $\mathtt{M}$ is spanned by the vectors $e_1+(g-1)f_1, t_1$ and $v$ for some root $v\in \Lambda$ orthogonal to $e_1+(g-1)f_1$ and $t_1$. 

Denote by 
$$\Lambda_{A_1}=\left<e_1-(g-1)f_1\right>\oplus U^{\oplus 2}\oplus E_8\oplus W_7$$
the orthogonal complement of $e_1+(g-1)f_1$ and $t_1$,  where $W_7=t_1^\perp\subset E_8$ is the orthogonal complement of $t_1$ in $E_8$. We are reduced to consider the $\widetilde{\rO}(\Lambda_{A_1})$-orbits of $v$.  Due to Proposition \ref{Eichler}, we know that the $\widetilde{\rO}(\Lambda_{A_1})$-orbits of $v$ is determined by $v^\ast$ in $\Lambda_{A_1}^\vee/\Lambda_{A_1}$.  Let $div(v)$ be the divisibility of $v$. Note that $div(v)$ is at most $2$  as $v^2=-2$. The discriminant group $\Lambda_{A_1}^\vee/\Lambda_{A_1}$ is isomorphic to $\ZZ/(2g-2)\ZZ\times \ZZ/2\ZZ$, which is generated by $\frac{(e_1-(g-1)f_1)}{2g-2}$ and $\frac{(t_1+2t_2)}{2}$. This gives three possibilities:

\begin{enumerate}[leftmargin=*]
    \item [i)]   $div(v)=1$ and $v^\ast=0$ in $\Lambda_{A_1}^\vee/\Lambda_{A_1}$;
    \item [ii)]  $div(v)=2$ and $v^\ast=\frac{1}{2}(e_1-(g-1)f_1)$ in $\Lambda_{A_1}^\vee/\Lambda_{A_1}$. By computing the norm of $v^\ast$, we get  $g\equiv 2 \mod 4$. In this case, one can take $v$ to be the vector $$e_1-(g-1)f_1+2(e_2+\frac{g-2}{4}f_2),$$
    where $(e_2,f_2)$ is the standard basis of the second hyperbolic lattice.
    
    \item [iii)] $div(v)=2$ and $v^\ast=\frac{1}{2}(e_1-(g-1) f_1+t_1+2t_2)$ in $\Lambda_{A_1}^\vee/\Lambda_{A_1}$. The normal computation shows that $g\equiv 3\mod 4$.  In this case, one can take $v$ to be the vector 
    $$ e_1-(g-1) f_1+t_1+2t_2+2(e_2+\frac{g+1}{4} f_2).$$
    
   \item [iiv)] $div(v)=2$ and $v^\ast=\frac{1}{2}(t_1+2t_2) $ in $\Lambda_{A_1}^\vee/\Lambda_{A_1}$, this is impossible because the norm of $v^\ast$ is $-\frac{1}{2}$ while $ \frac{1}{2}(t_1+2t_2)$ is $-\frac{3}{2}$. 
\end{enumerate}

For $\cH^g_{A_2}$, the computation is similar. We just need to classify the $\widetilde{\rO}(\Lambda_g)$-orbits of lattices spanned by $e_1+(g-1)f_1$, $t_1$ and $v$ for some $v$ orthogonal to $e_1+(g-1)f_1$, $v^2=-2$ and $t_1\cdot v=1$. The vector $t_1+2v$ lies in $\Lambda_{A_1}$ with norm $-6$. One can see that  the divisibility of $t_1+2v$ in $\Lambda_{A_1}$ is either $2$ or $6$. Among all cases, there is only one possibility:
$$(t_1+2v)^\ast= \frac{t_1+2t_2}{2}~\hbox{in $\Lambda_{A_1}^\vee/\Lambda_{A_1}$}.$$
One can take $v=t_2$ in this case. The rest of assertions follows easily. 
\end{proof}

Furthermore, one can consider the binodal locus and cuspidal locus on the moduli space of lattice polarized K3 surfaces. The following result is straightforward. 
\begin{proposition}\label{lattice-cuspidal-binodal}
Let $\cH_{U,A_2}$ and $\cH_{U,A_{1,1}}$ be the cuspidal and binodal locus of the moduli space of $U$-lattice polarized K3 surface respectively. Then $\cH_{U,A_2}$  and $\cH_{U,A_{1,1}}$ are both irreducible  and they are a multiple of special cycles defined in Example \ref{Modularity-unimodular}.
\end{proposition}
\begin{proof}
According to Example \ref{Elliptic}, $U^\perp\cong U^{\oplus 2} \oplus E_8^{\oplus 2} $ is unimodular. Using the period map,  we can write $\cH_{U,A_2}$ and $\cH_{U,A_{1,1}}$  as a  multiple of special cycles via \eqref{specialcycle} and \eqref{eq：heegnercycle}:  
\begin{equation}
    \cH_{U,A_2}=\frac{1}{2}Z(\left(\begin{array}{ccc} -2  & 1 \\ 1 & -2 \end{array}\right) ) ~\hbox{and}~\cH_{U,A_{1,1}}=\frac{1}{2}Z(\left(\begin{array}{ccc} -2  & 0 \\ 0 & -2 \end{array}\right) ).
\end{equation}
The multiplicity $\frac{1}{2}$ comes from the fact: if $\bv=(v_1,v_2)\in \Omega_{A_2}$ (resp.~ $\Omega_{A_{1,1}}$), then the vector $\bv'=(v_2,v_1)$ also occurs in $\Omega_{A_2}$ (resp.~$\Omega_{A_{1,1}}$). The NL-cycle is taking the sum over the collection of the set consisting of components $\{v_1,v_2\}$ while the special cycle is taking the sum over the collection of vectors $(v_1,v_2)$. 

  The irreducibility of them  follows from Nikulin's result \cite[Theorem 1.14.4]{Ni79}, which shows that the primitive embeddings of  $\left(\begin{array}{ccc} -2  & 1 \\ 1 & -2 \end{array}\right) $ and $\left(\begin{array}{ccc} -2  & 0 \\ 0 & -2 \end{array}\right) $  into $U^{\oplus 2} \oplus E_8^{\oplus 2}$ are unique. 
\end{proof}

\section{Deligne-Beilinson cohomology of stacks}\label{Sec4}
In this section, we review the theory of DB cohomology of algebraic varieties and extend it to the case of Deligne-Mumford stacks of finite type over $\CC$. Throughout this section, all schemes and stacks are assumed to be of finite type over $\CC$  and all (hyper)cohomologies are taken on analytic topoi unless otherwise stated.

\subsection{DB cohomology of smooth varieties and arbitrary varieties}
\begin{definition}(\cite{Bei84, EV88}) Let $X$ be a smooth complex algebraic variety, and $j:X\hookrightarrow \overline{X}$ be a good compactification with boundary divisor $Z$. The Deligne-Beilinson cohomology of $X$ is defined as the analytic hypercohomology 
$$\rH^q_{\DB}(X,A(p))= \rH^q (\overline{X}_{\an},A(p)_{\DB, (\overline{X}, X)})$$
where 
$$A(p)_{\DB, (\overline{X}, X)}= \mathrm{Cone} (R j_* A(p) \oplus \Omega_{X}^{\geq p}(\log Z)\xrightarrow{\epsilon-\iota} R f_* \Omega_{X}^\bullet)[-1]$$
is a complex of analytic sheaves on $\overline{X}$, called the Deligne-Beilinson complex of the pair $(\overline{X}, X)$. Here, $A$ is a subring of $\RR$ and $A(p)$ is the constant sheaf with value $(2\pi i)^p A$; the complex $\Omega^\bullet_X$ is the de Rham complex of holomorphic forms on $X$
and $\Omega^{\geq p}_X(\log Z)$ is the brutal truncation of the complex of meromorphic 
forms on $X$ with at most logarithmic poles along $Z$;  the maps $\epsilon$ and $\iota$ are the natural ones.
\end{definition}
Moreover, there is a natural map 
$\rH^{2p}_{\DB}(X,A(p))\to \rH^{2p}(X,A(p))$ and a 
DB-cycle class map 
$$\cl_{\DB}:\mathrm{CH}^p(X)\to \rH^{2p}_{\DB}(X,\ZZ(p))$$
defined by lifting the Betti fundamental classes. When $X$ is smooth and projective, they are compatible with the Abel-Jacobi map
\begin{equation}
    \xymatrix{
    0\ar[r]&\CH^p_{hom}(X)\ar[r]\ar[d]_{AJ}&\CH^p(X)\ar[rd]^{\cl_{\rm B}}\ar[d]_{\cl_{\DB}}& \\
    0\ar[r]&J^{p}(X)\ar[r]&\rH^{2p}_{\DB}(X,\ZZ(p)) \ar[r]& \rH^{2p}(X,\ZZ(p))
    }
\end{equation}

The notion of DB cohomology can be extended to the case of arbitrary varieties and even more generally, arbitrary (separated) simplicial schemes via the theory of cohomological descent ({\it cf.}~\cite[\S 5]{EV88}). We briefly recall the constructions. For a separated simplicial scheme $X_\bullet$, one can find a diagram
\begin{equation}\label{resolution_of_coh_desc}
	\xymatrix
	{
		U_{\bullet}\ar[r]^{j}\ar[d]_{p} &\overline{U}_{\bullet}\\
		X_{\bullet}
	}
\end{equation}
where $p$ is (a proper hypercovering, hence) of cohomological descent, i.e.  the cohomologies of $X_{\bullet}$ and of $U_{\bullet}$ are identified via $p$, and $j$ is a good compactification \cite[8.3.2]{De74}. Then the Deligne-Beilinson cohomology of $X_\bullet$ is defined to be the analytic hypercohomology
$$\rH^q_{\DB}(X_{\bullet},A(p))= \rH^q (\overline{U}_{\bullet, \an}, A(p)_{\DB, (\overline{U}_\bullet, U_\bullet)})$$
where $A(p)_{\DB, (\overline{U}_\bullet, U_\bullet)}$ is the Deligne-Beilinson complex of $(\overline{U}_\bullet, U_\bullet)$.

Many properties still hold in such a general case (see \cite[5.4]{EV88}). Among them we state a long exact sequence involving the DB cohomology and Betti cohomology of $X_\bullet$, which is the most important for our purpose. Recall \cite[Definition 8.3.4]{De74} that the mixed Hodge structure on $\rH^{q}(X_\bullet, \ZZ)$ is defined by transport from that on $\rH^{q}(U_\bullet, \ZZ)$ via the isomorphism $p^*$ (for any choice of diagram (\ref{resolution_of_coh_desc})). We denote by  $\{F^p\rH^q(X_\bullet,\CC)\}$  the Hodge filtration  on $\rH^q(X_\bullet,\CC)$. Then we have 
\begin{proposition}(\cite[2.10(b) and 5.4]{EV88})
Let $X_\bullet$ be a (separated) simplicial scheme, then there exists a long exact sequence
    \begin{equation}\label{long_exact_sequence_involving_DB_cohomology}
        \cdots \to \rH^q_{\DB}(X_\bullet, A(p)) \to \rH^q(X_\bullet, A(p)) \to \rH^q(X_\bullet, \CC)/F^p\rH^q(X_\bullet, \CC) \to \cdots
    \end{equation}
and this sequence is contravariantly functorial for any morphism $X'_\bullet\to X_\bullet$.
\end{proposition}

\subsection{Cohomology on Deligne-Mumford stacks}
Let $\scX$ be a Deligne-Mumford stack of finite type over $\CC$. We denote by $\Et(\scX)$ the \'etale site of $\scX$, whose objects are \'etale morphisms $(U\to \scX)$, the morphisms from $(U'\to \scX)$ to $(U\to \scX)$ are morphisms of schemes $(U'\to U)$ such that the corresponding diagram with base $\scX$ is (2-)commutative, and the coverings of $(U\to \scX)$ are those $((U_i\to \scX)\to (U\to \scX))_i$ for which $(U_i\to U)_i$ is an \'etale covering of $U$. The resulting topos is denoted by $\scX_{\et}$, and the sheaves are called \'etale sheaves on $\scX$.
\begin{remark}\label{separated_sub_site}
Consider the full subcategory ${\rm \acute{E}}{\rm t}^{\rm sep}(\scX)$ of $\Et(\scX)$ whose objects are those $(U\to \scX)$ for which $U$ is \textit{separated}. It is not difficult to see that ${\rm \acute{E}}{\rm t}^{\rm sep}(\scX)$ is naturally a site whose resulting topos is equivalent to $\scX_{\et}$.
\end{remark}

Let $\cF$ be an \'etale sheaf on $\scX$, then one can define the localized topos $\scX_{\et}/{\cF}$ which is simply the category $\scX_{\et}$ over $\cF$. In the  case where $\cF$ is given by an \'etale morphism $U\to \scX$, it is also denoted by $\scX_{\et}/U$, and the fact that it is indeed a topos follows from the following description: let $\Et(\scX)/U$ be the localized site, then the topos associated to $\Et(\scX)/U$ is equivalent to $\scX_{\et}/U$. Moreover, it is not hard to see  that  $\scX_{\et}/U$ is equivalent to  $U_{\et}$, the \'etale topos of $U$.

We now review the method to ``compute" the cohomology using an atlas. Let $\pi:X\to \scX$ be an \'etale atlas of $\scX$, i.e. $X$ is a scheme and $\pi$ is an \'etale surjective morphism. To any such $\pi$, we can associate its 0-coskeleton $$X_{\bullet}:=\cosk_0(\pi),$$ which is a simplicial scheme, and an augmentation $e:X_\bullet\to \scX$ (\cite[5.1.4]{De74}, \cite[13.5]{LMB}). Note that replacing $X$ by $\amalg_i U_i$ where $(U_i)_i$ is a Zariski open covering of $X$ consisting of affine schemes (and finite in number), we may always assume that $X$ is separated, and $X_{\bullet}$ is then a separated simplicial scheme.

Since $X_{\bullet}$ is a simplicial \'etale sheaf on $\scX$, there is a localized topos $\scX_{\et}/X_{\bullet}$  and a morphism of topoi $$e_{\et}:\scX_{\et}/X_{\bullet}\to \scX_{\et}$$ ({\it cf}.~\cite[\S 2.4.5 and \S 2.4.11]{Olsson}). 
It follows from definition that the localized topos $\scX_{\et}/X_{\bullet}$ is indeed equivalent to the \'etale topos $X_{\bullet, \et}$ of $X_{\bullet}$ defined by Deligne ({\it cf.}~\cite[12.4]{LMB}). Recall that  sheaves in $X_{\bullet, \et}$ are those
$$((\cF_n)_{n\geq 0}, (\theta_{\delta})_{\delta:[n']\to [n]})$$
where for each $n$, $\cF_n$ is an \'etale sheaf on $X_n$, and for each $\delta:[n']\to [n]$, $\theta_{\delta}: \cF_{n'}\to \delta_*\cF_n$ is a morphism of \'etale sheaves on $X_{n'}$ (here we denote by $\delta$ also the morphism $X_n\to X_{n'}$ associated to $\delta:[n']\to [n]$), satisfying suitable compatibility conditions on composition of transition maps.

Since $\pi$ is \'etale, surjective and of finite type, hence a covering in the canonical topology, $e_{\et}:X_{\bullet, \et} \to \scX_{\et}$ is of cohomological descent ({\it cf.}~\cite[2.4.16]{Olsson}), i.e., let $D^+(\scX_{\et}, \ZZ)$ be the derived category of bounded below complexes of $\ZZ$-modules in $\scX_{\et}$, then for any $\cF\in D^+(\scX_{\et}, \ZZ)$, the morphism
\begin{equation*}
    \cF \to Re_{\et,*}e^{-1}_{\et} \cF
\end{equation*}
in $D^+(\scX_{\et}, \ZZ)$ is an isomorphism. 
Note that $e^{-1}_{\et}(\cF)$ has the following concrete description: it is the complex of sheaves in $X_{\bullet, \et}$ where the degree $n$ component is simply the restriction of $\cF$ to $X_{n, \et}$, and the transition maps are those induced by $X_{\bullet}$. Under these notations we have therefore an isomorphism

\begin{equation}\label{compute_cohomology_of_etale_sheaves}
    \rH^q(\scX_{\et}, \cF)\to \rH^q(X_{\bullet, \et}, e^{-1}_{\et}(\cF))
\end{equation}
\begin{remark}
In \cite[\S 9.2]{Olsson},  the constructions above are carried out for algebraic stacks, using the lisse-\'etale site $\rm{Lis}$-$\Et(\scX)$ of an algebraic stack $\scX$. Recall that $\rm{Lis}$-$\Et(\scX)$ is the category of \textit{smooth} morphisms over $\scX$, where coverings of a \textit{smooth} morphism $U\to \scX$ are those $((U_i\to \scX)\to (U\to \scX))_i$ for which $(U_i\to U)_i$ is an \textit{\'etale} covering of $U$. The resulting topos is denoted by $\scX_{\rm{lis}\text{-}\et}$.

In general cases, there are technical issues on comparing $\scX_{{\rm lis}\text{-}\et}/X_{\bullet}$ and $X_{\bullet, \et}$. For  Deligne-Mumford stacks, however, these issues disappear.
\end{remark}

There is a similar picture on the analytic side. For an \'etale atlas $\pi:X\to \scX$, the analytification $\pi^{\an}:X^{\an}\to \scX^{\an}$ is also 
a covering in the canonical topology (of the analytic topos $\scX_{\an}$). Therefore for $\cF\in D^+(\scX_{\an}, \ZZ)$, we have also an isomorphism
\begin{equation}\label{compute_cohomology_of_analytic_sheaves}
    \rH^q(\scX_{\an}, \cF)\to \rH^q(X_{\bullet, \an}, e^{-1}_{\an}(\cF))
\end{equation}

We now give the construction of mixed Hodge structures on the cohomology groups of $\scX$ using an atlas.

\begin{lemma}\label{comparison_between_two_atlases}
	Let $\pi_1:X_1\to \scX$ and $\pi_2:X_2\to \scX$ be two \'etale atlases, where $X_1$ and $X_2$ are separated schemes, and $e_1:(X_1)_\bullet\to \scX$ and $e_2:(X_2)_\bullet\to \scX$ are augmentations from the associated $0$-coskeletons to $\scX$. Then for $i=1, 2$, the transport of mixed Hodge structures from those of $\rH^{q}((X_i)_\bullet, \ZZ)$ via $e_i^*$ to $\rH^{q}(\scX, \ZZ)$ are the same. 
\end{lemma}

\begin{proof}
    Let $X_3$ be the (2-)fiber product $X_1\times_{\scX}X_2$. We have the following pull-back diagram
    \begin{equation*}
	\xymatrix
	{
		X_3\ar[r]^{\widetilde{\pi_2}}\ar[d]_{\widetilde{\pi_1}} &X_1\ar[d]^{\pi_1}\\
		X_2\ar[r]_{\pi_2} & \scX
	}
	\end{equation*}
    where $\widetilde{\pi_1}$ and $\widetilde{\pi_2}$ are \'etale surjective. In the associated commutative diagram
    \begin{equation*}
	\xymatrix
	{
		(X_3)_{\bullet}\ar[r]^{\widetilde{e_2}}\ar[d]_{\widetilde{e_1}} &(X_1)_{\bullet}\ar[d]^{e_1}\\
		(X_2)_{\bullet}\ar[r]_{e_2} & \scX
	}
	\end{equation*}
    $\widetilde{e_1}$ and $\widetilde{e_2}$ (the morphisms between $0$-coskeletons induced by  $\widetilde{\pi_1}$ and $\widetilde{\pi_2}$, respectively) are then \'etale hypercoverings, hence of cohomological descent. Therefore in the following commutative diagram of cohomology groups
    \begin{equation}\label{diagram_of_iso_1}
	\xymatrix
	{
		\rH^{q}((X_3)_\bullet, \ZZ) &\rH^{q}((X_1)_\bullet, \ZZ) \ar[l]_{(\widetilde{e_2})^*}\\
		\rH^{q}((X_2)_\bullet, \ZZ) \ar[u]^{(\widetilde{e_1})^*}& \rH^{q}(\scX, \ZZ) \ar[l]^{e_2^*} \ar[u]_{e_1^*}
	}
	\end{equation}
    $(\widetilde{e_1})^*$ and $(\widetilde{e_2})^*$ are (morphisms of mixed Hodge structures and isomorphisms of the underlying $\ZZ$-modules, hence) isomorphisms of mixed Hodge structures
    , so for $i=1,2$ the transport of mixed Hodge structures from those of $\rH^{q}((X_i)_\bullet, \ZZ)$ via $e_i^*$ to $\rH^{q}(\scX, \ZZ)$ are both equal to the transport from that of $\rH^{q}((X_3)_\bullet, \ZZ)$.
\end{proof}

From this, a mixed Hodge structure on $\rH^{q}(\scX, \ZZ)$ is defined, as the transport from that of $\rH^{q}(X_\bullet, \ZZ)$ via $e^*$ using any \'etale atlas.  This gives in particular the Hodge filtration $\{F^p\rH^{q}(\scX, \CC)\}$ of $\rH^{q}(\scX, \ZZ)$ and (\ref{diagram_of_iso_1}) becomes a commutative diagram of isomorphisms of mixed Hodge structures. 

\subsection{DB cohomology of Deligne-Mumford stacks: via atlases}
As before, we let $\scX$ be a Deligne-Mumford stack of finite type over $\CC$.  The first result is  

\begin{lemma}
With the same notations as in Lemma \ref{comparison_between_two_atlases}, there is an isomorphism
	\begin{equation}
	    \rH^q_{\DB}((X_1)_\bullet, A(p))\cong \rH^q_{\DB}((X_2)_\bullet, A(p))
	\end{equation}
	for any $q$.
\end{lemma}
\begin{proof}
The proof is similar to that of  \cite[5.2]{EV88}. Keep notations the same as in the proof of Lemma \ref{comparison_between_two_atlases}. The morphism $\widetilde{e_2}$ induces a morphism of two long exact sequences (\ref{long_exact_sequence_involving_DB_cohomology}):
    \begin{equation*}
		\xymatrix
		{
			\cdots\ar[r]& \rH^q_{\DB}((X_1)_\bullet, A(p))\ar[r]\ar[d]^{(\widetilde{e_2})^*_{\DB}}&\rH^q((X_1)_\bullet, A(p))\ar[r]\ar[d]^{(\widetilde{e_2})^*}&\rH^q((X_1)_\bullet, \CC)/F^p\rH^q((X_1)_\bullet, \CC)\ar[r]\ar[d]^{\overline{(\widetilde{e_2})^*}}&\cdots\\
			\cdots\ar[r]& \rH^q_{\DB}((X_3)_\bullet, A(p))\ar[r]&\rH^q((X_3)_\bullet, A(p))\ar[r]&\rH^q((X_3)_\bullet, \CC)/F^p\rH^q((X_3)_\bullet, \CC)\ar[r]&\cdots
		}
	\end{equation*}
	
Since $(\widetilde{e_2})^*$ is an isomorphism of mixed Hodge structures, $\overline{(\widetilde{e_2})^*}$ is an isomorphism, hence $(\widetilde{e_2})^*_{\DB}$ is also an isomorphism. Similarly, the induced map $$(\widetilde{e_1})^*_{\DB}: \rH^q_{\DB}((X_2)_\bullet, A(p))\to \rH^q_{\DB}((X_3)_\bullet, A(p))$$ is an isomorphism as well. We conclude that $((\widetilde{e_1})^*_{\DB})^{-1} \circ (\widetilde{e_2})^*_{\DB}$ gives the desired isomorphism.
\end{proof}
\begin{remark}\label{transition_iso}
Suppose that there is a morphism $f:X_1\to X_2$ such that the following diagram is commutative:
    \begin{equation*}
	\xymatrix
	{
		 &X_1\ar[d]^{\pi_1}\ar[dl]_{f}\\
		X_2\ar[r]_{\pi_2} & \scX
	}
	\end{equation*}
Let $f_{\bullet}:(X_1)_\bullet\to (X_2)_\bullet$ be the morphism induced by $f$. Then the same argument shows that $$((f_{\bullet})^*_{\DB})^{-1}: \rH^q_{\DB}((X_1)_\bullet, A(p))\to \rH^q_{\DB}((X_2)_\bullet, A(p))$$ is an isomorphism, denoted by $\eta_{f}$.
\end{remark}

\begin{definition}\label{def_of_DB}
	Let $\scX$ be a Deligne-Mumford stack of finite type over $\CC$. Fix any \'etale atlas $\pi:X\to \scX$ (with $X$ separated), we define the Deligne-Beilinson cohomology of $\scX$ to be
	\begin{equation}
	    \rH^q_{\DB}(\scX, A(p)):=\rH^q_{\DB}(X_\bullet, A(p))
	\end{equation}
	where $X_\bullet$ is the $0$-coskeleton of $\pi$. 
\end{definition}

Moreover, the construction above is functorial. Let us fix for any $\scX$ an arbitrary  \'etale atlas $\pi:X\to \scX$ (with $X$ separated). Let $f:\scX'\to \scX$ be a morphism. Take the \'etale atlas $\pi':X'\to \scX'$ of $\scX'$ we have fixed, and let $\widetilde{\pi}: X'_{(X)}\to X'$ be the pull-back of $\pi$ along $f\circ \pi'$, which fits into the following commutative diagram 
\begin{equation*}
    \xymatrix
		{
			X'_{(X)} \ar[rr]^{\widetilde{f}} \ar[d]_{\widetilde{\pi}} \ar[dr] & & X \ar[d]^{\pi}\\
			X'\ar[r]_{\pi'} & \scX'\ar[r]_{f} & \scX.
		}
\end{equation*}
Then we have 

\begin{proposition}
Set
\begin{equation}
    f^*_{\DB}:=\eta_{\widetilde{\pi}} \circ (\widetilde{f_{\bullet}})^*_{\DB}: \rH^q_{\DB}(\scX, A(p))\to \rH^q_{\DB}(\scX', A(p))
\end{equation}
where $\eta_{\widetilde{\pi}}$ is defined in Remark \ref{transition_iso} and $\widetilde{f_{\bullet}}$ is the morphism of simplicial schemes induced by $\widetilde{f}$. Then  
the assignment $\scX\to \rH^q_{\DB}(\scX, A(p))$, $f\to f^*_{\DB}$ defined above is contravariantly functorial for any morphism $\scX'\to\scX$.
\end{proposition}
\begin{proof}
Suppose that we are also given $f': \scX''\to \scX'$. Let $\pi'':X''\to \scX''$ be the \'etale atlas of $\scX''$ we have fixed. Consider the following diagram:
\begin{equation*}
    \xymatrix
		{
			& X''_{(X)}\ar[rr] \ar[dd]& &\scX''_{(X)} \ar[rr]\ar[dd] & & \scX'_{(X)}\ar[r]\ar[dd]& X \ar[dd]^{\pi}\\
			X''_{(X)}\times_{X''}X''_{(X')}\ar[ur]\ar[rrrr]\ar[dd]& & & & X'_{(X)}\ar[ur]\ar[dd]& &\\
			& X''\ar[rr]^{\pi''}& &\scX''\ar[rr]^{f'} & & \scX'\ar[r]^{f} & \scX\\
			X''_{(X')}\ar[ur]\ar[rr]& & \scX''_{(X')}\ar[ur]\ar[rr] & & X'\ar[ur]_{\pi'} & &
		}
\end{equation*}
where all squares are pull-backs. Denote the arrow $X''_{(X)}\times_{X''}X''_{(X')}\to X''$ by $g$ and the arrow $X''_{(X)}\times_{X''}X''_{(X')}\to X$ by $h$.
It follows that $(f\circ f')^*_{\DB}$ and $(f')^*_{\DB}\circ (f)^*_{\DB}$ are both naturally isomorphic to $\eta_{g}\circ (h_{\bullet})^*_{\DB}$.

\end{proof}

\begin{proposition}
Let $\scX$ be a Deligne-Mumford stack of finite type over $\CC$. Then there is an exact sequence
	\begin{equation}\label{long_exact_seq_for_stacks}
		\xymatrix
		{
			\cdots\ar[r]&\rH^q_{\DB}(\scX, A(p))\ar[r]&\rH^q(\scX, A(p))\ar[r]&\rH^q(\scX, \CC)/F^p\rH^q(\scX, \CC)\ar[r]&\cdots
		}
	\end{equation}
\end{proposition}
\begin{proof}
Take $\pi: X\to \scX$ as in Definition \ref{def_of_DB}, then the assertion follows from the construction of $\rH^q_{\DB}(\scX, A(p))$ and the exact sequence \eqref{long_exact_sequence_involving_DB_cohomology} for the 0-coskeleton $X_{\bullet}$
\begin{equation*}
		\xymatrix
		{
			\cdots\ar[r]&\rH^q_{\DB}(X_{\bullet}, A(p))\ar[r]&\rH^q(X_{\bullet}, A(p))\ar[r]&\rH^q(X_{\bullet}, \CC)/F^p\rH^q(X_{\bullet}, \CC)\ar[r]&\cdots
		}
	\end{equation*}
\end{proof}

\subsection{DB cohomology for smooth Deligne-Mumford stacks: via DB complexes}
In this subsection we give an alternative definition of DB cohomology for \textit{smooth} Deligne-Mumford stacks, which will be used to define the DB-cycle class maps.

So let us assume now that $\scX$ is \textit{smooth}. In particular, for every \'etale morphism $U\to \scX$, $U$ must be smooth. Consequently if $X\to \scX$ is an \'etale atlas with $X$ separated (and smooth), its 0-coskeleton is a smooth separated simplicial scheme.

Let $Sm/\CC$ be the category of smooth separated schemes.  In \cite{Bei84}, Beilinson showed that there is a complex of (big) sheaves $\QQ(p)_{\DB , \Zar}$ on $(Sm/\CC)_{\Zar}$ which computes DB cohomology. In particular, for all smooth separated simplicial scheme $X_{\bullet}$, we have an isomorphism
\begin{equation*}
    \rH^q_{\DB}(X_{\bullet}, \QQ(p)) \cong \rH^q(X_{\bullet, \Zar}, \QQ(p)_{\DB, \Zar}|_{X_{\bullet, \Zar}}),
\end{equation*}
(See also \cite[\S 5.5]{EV88}).

Let $\QQ(p)_{\DB, \et}$ be the sheafification of $\QQ(p)_{\DB , Zar}$ in the \'etale topology. It has been proved that the complex of \'etale sheaves $\QQ(p)_{\DB, \et}$ also computes the DB cohomology. In particular, for all smooth separated simplicial schemes $X_{\bullet}$, we have
\begin{equation}\label{etale_descent_for_DB}
\rH^q_{\DB}(X_{\bullet}, \QQ(p)) \cong \rH^q(X_{\bullet, \et}, \QQ(p)_{\DB, \et}|_{X_{\bullet, \et}}),
\end{equation}
({\it cf}.~ \cite[Thm. 2.5 (i)]{Koh19}). This holds even in $\ZZ$-coefficients.

Now since $\scX$ is smooth, we can restrict $\QQ(p)_{\DB, \et}$ to $\scX_{\et}$ (see Remark \ref{separated_sub_site}, and note that the objects of ${\rm \acute{E}}{\rm t}^{\rm sep}(\scX)$ are all smooth).

\begin{proposition}
Let $\scX$ be a \textit{smooth} Deligne-Mumford stack. Then there is an isomorphism
    \begin{equation}
        \rH^q_{\DB}(\scX, \QQ(p)) \cong \rH^q(\scX_{\et}, \QQ(p)_{\DB, \et}|_{\scX_{\et}})
    \end{equation}
\end{proposition}
\begin{proof}
    Take an \'etale atlas $\pi:X\to \scX$. By definition $\rH^q_{\DB}(\scX, \QQ(p))=\rH^q_{\DB}(X_\bullet, \QQ(p))$, where $X_\bullet$ is the 0-coskeleton of $\pi$. By (\ref{etale_descent_for_DB}),  it is isomorphic to $\rH^q(X_{\bullet, \et}, \QQ(p)_{\DB, \et}|_{X_{\bullet, \et}})$, and by (\ref{compute_cohomology_of_etale_sheaves}) also isomorphic to $\rH^q(\scX_{\et}, \QQ(p)_{\DB, \et}|_{\scX_{\et}})$.
\end{proof}

\subsection{DB-cycle class maps for smooth Deligne-Mumford stacks}
Now we proceed to define the DB-cycle class map from the Chow group of a smooth stack to the DB cohomology group. 

Let $\scX$ be a \textit{smooth} Deligne-Mumford stack. For each integer $p$, let $\CH^p(\scX)_\QQ$ be the $p$-th rational Chow group of $\scX$ defined in  \cite{Gil84, Vis89}. Let $\QQ(p)$ denote the motivic complex of weight $p$ with rational coefficients (see \cite[Definition 3.1]{MVW06}), which is a complex of sheaves on $(Sm/\CC)_{\et}$. Then we have

\begin{proposition} ({\it cf}.~\cite[Theorem 2 (i)]{JoshuaI})
    There is a canonical isomorphism $$\CH^p(\scX)_\QQ\cong \rH^{2p}(\scX_{\et}, \QQ(p)|_{\scX_{\et}}).$$
\end{proposition}

According to \cite[Thm 5.2 (i)]{Koh19} (and see also \cite[Chapter 19]{MVW06}), there is a map 
$$\QQ(p)\to \QQ(p)_{\DB, \Zar},$$
such that the induced map of the $2p$-th Zariski hypercohomologies, restricting to $X_{\Zar}$ for any smooth separated scheme $X$, can be identified with the classical DB-cycle class map $$\cl_{\DB}:\mathrm{CH}^p(X)_{\QQ}\to \rH^{2p}_{\DB}(X,\QQ(p)).$$ Sheafifying the above map in the \'etale topology (note that $\QQ(p)$ is a complex of sheaves in the \'etale topology), we get a map
\begin{equation}\label{motivic_to_DB}
    \QQ(p)\to \QQ(p)_{\DB, \et}.
\end{equation}

This yields 
\begin{definition}
    Let $\scX$ be a smooth Deligne-Mumford stack of finite type over $\CC$. The DB-cycle class map for $\scX$ with $\QQ$-coefficients
    $$\cl_{\DB}:\mathrm{CH}^p(\scX)_{\QQ}\to \rH^{2p}_{\DB}(\scX,\QQ(p))$$
    is defined by taking the $2p$-th \'etale hypercohomology of the map (\ref{motivic_to_DB})
    restricting to $\scX_{\et}$.
\end{definition}

We have 
\begin{proposition}
If $\scX$ is a smooth scheme, the cycle class map $\cl_{\DB}$
agrees with the Deligne-Beilinson cycle class map defined in \cite{Blo86}. 
\end{proposition}
\begin{proof}
In this case we simply take the identity $X:=\scX \xrightarrow{id} \scX$ as an atlas, the $0$-coskeleton is then the constant simplicial scheme $X$, and the result follows from (\ref{etale_descent_for_DB}) and the description above of the DB-cycle class map for $X$.
\end{proof}

\section{GFC for moduli stack of oriented algebraic K3 surfaces}\label{Sec5}

In this section, we will study the separated locus of the moduli stack of oriented quasi-polarized K3 surfaces. 

\subsection{Oriented K3 surfaces}Following \cite{GHS07}, we first introduce the spin structures on a K3 surface. 
\begin{definition}
    {\it A spin structure} on a quasi-polarized K3 surface $(S, L)$ of genus $g$ is a choice $\omega$ of the two generators of $$\Hom_{\ZZ}(\det(\rH^2_{\rm prim}(S, \ZZ)), \det \Lambda_g).$$  The triple $(S, L, \omega)$ is called an {\it oriented  quasi-polarized K3 surface of genus $g$}. A morphism between two oriented quasi-polarized K3 surfaces $(S, L, \omega)$ and  $(S', L', \omega')$  is a morphism $f:S'\to S$ satisfying $f^*(L)=L'$ and $\omega'=\omega\circ\det(f^*)$.
\end{definition}

Let $\widetilde{\scF_g}$ be the moduli stack of quasi-polarized oriented K3 surfaces of genus $g$.  The natural forgetful functor
$$u:\widetilde{\scF_g}\to \scF_{g},$$
is \'etale of degree $2$.  It follows that $\widetilde{\scF}_g$ is a smooth Deligne-Mumford stack. Let $\widetilde{\cF_g}$ be the coarse moduli space of $\widetilde{\scF}_g$. Then it is isomorphic to a Shimura variety of orthogonal type. 
 \begin{proposition}[\cite{GHS07}]
There is an isomorphism 
$\widetilde{\cF}_g\to \widetilde{\rS\rO}(\Lambda_g)\backslash \rD_{\Lambda_g}$ and the forgetful map on coarse moduli spaces  $\widetilde{\cF}_g\to \cF_g$ is a double covering branched over $\cH^g_{0,-2}$.
\end{proposition}

The stack $\widetilde{\scF}_{g}$ remains non-separated, but it behaves slightly better than $\scF_g$ in the sense that it admits a larger separated locus.

\subsection{Separated locus of $\widetilde{\scF}_g$} Let us consider the open substack $\widetilde{\scF}_g^s$ of oriented quasi-polarized K3 surfaces of genus $g$ containing at most one exceptional $(-2)$-curve. A key observation is 
\begin{proposition}\label{seploci}
The open substack  $\widetilde{\scF}_g^s$ is a separated smooth Deligne-Mumford stack. 
\end{proposition}
\begin{proof}
Let $R$ be a discrete valuation ring over $\CC$ and $K$ its fraction field. To prove the separatedness,  it suffices to show that for any two families 
$(\cX, \cL, \omega) \to \Spec R$ and  $(\cX',\cL',\omega')\to \Spec R$ of quasi-polarized oriented K3 surfaces of genus $g$ with at most one exceptional $(-2)$-curve,  if there is an isomorphism 
\begin{equation}
  \phi_K:  (\fX_K, \cL_K, \omega_K)\cong (\fX'_K,\cL'_K, \omega'_K)
\end{equation}
of quasi-polarized oriented K3 surfaces over the generic point, then $\phi_K$ can be uniquely extended to an isomorphism over $R$. There are three possibilities: 

(1) if $\cL$ is relative ample, one can apply \cite[Theorem 1]{MM64} to the pairs $(\cX,\cL)$  and $(\cX,\cL')$. It is not hard to such extension must preserve the orientation. 

(2) if $\cL_K$ is not ample, then by our assumption, there exist unique families of exceptional $(-2)$-curves on $\cX\to \Spec R$ and $\cY\to \Spec R$, denoted by $\cC\to \Spec R$ and $\cC'\to \Spec R$ respectively.  One can apply  \cite[Theorem 1]{MM64} to the pairs  $(\cX,\cL-\cC)$ or $(\cX', \cL'-\cC')$ to obtain the extension.

(3) if $\cL_K$ is ample, but the restriction $\cL_0=\cL|_{\cX_0}$ is not ample. Then $\cX_0$ contains an exceptional $(-2)$-curve $C$.  If the isomorphism of the pairs $(\cX_K,\cL_K)\cong (\cX'_K,\cL_K')$ can not be extended to $\Spec R$,  then the birational map $(\cX,\cL)\dashrightarrow (\cX', \cL') $ is an elementary transform associated to $C$ (cf.~\cite[Theorem 2]{BR75}). However, the monodromy group of a single elementary transform acts on $\rH^2(\cX_\eta,\ZZ)$ as a reflection, which does not preserve the orientation of the generic fiber. This contradicts to our assumption.
\end{proof}

Let $\widetilde{\pi}: \widetilde{\scS}_g^s\to \widetilde{\scF}_g^s$ be the universal family. Then $\widetilde{\scS}_g^s$ is also a separated Deligne-Mumford stack (non quasi-projective) and we denote by $\CH^2(\widetilde{\scS}_g^s)$ the rational Chow group.   We can state a stacky version of O'Grady's generalized Franchetta conjecture (GFC) for universal oriented K3 surfaces.
\begin{conjecture}[GFC for Oriented K3 surfaces]\label{conj:stackgenfranchet}
For any cycle class $\alpha\in\CH^2(\widetilde{\scS}_g^s)$, there exists $m\in\QQ$ and an open substack $U\subseteq \widetilde{\scF}_g^s$ such that the restriction of  $\alpha-mc_2(\cT_{\widetilde{\pi}}) $ to $\widetilde{\pi}^{-1}(U)$ is zero. 
\end{conjecture}

Certainly,  Conjecture \ref{conj:stackgenfranchet} will automatically imply Conjecture \ref{conj1} via applying the push-forward map. The cohomological version of Conjecture \ref{conj:stackgenfranchet}  has also been confirmed in \cite{BL17}.

\section{Deligne-Beilinson cohomology of $\widetilde{\scS}_g^s$}\label{Sec6}
In this section, we construct some two-dimensional testing families of polarized K3 surfaces and calculate the  NL-numbers. Together with the cohomological generalized Franchetta conjecture, this leads to a proof of Theorem \ref{mainthm}.  

\subsection{Vanishing of odd cohomology groups}
 
Based on the vanishing result in Theorem \ref{BL-vanishing}, let us first give a sufficient condition for the vanishing of $\rH^3(\widetilde{\scS}_g^s,\CC)$.
\begin{lemma}\label{vanishing}
Let $\widetilde{\cH}_{A_{1,1}}$ and $\widetilde{\cH}_{A_{2}}$ be the preimages of $\cH_{A_{1,1}}$ and $\cH_{A_2}$ in $\widetilde{\cF}_g$ respectively. Suppose the fundamental classes of irreducible components of $\widetilde{\cH}_{A_{1,1}}$ and $\widetilde{\cH}_{A_{2}}$ are linearly independent in $\rH^4(\widetilde{\cF}_g, \QQ)$, then $\rH^3(\widetilde{\cF}_g^s,\CC)=0$ and $\rH^1(\widetilde{\cF}_g^s,\underline{\Lambda_g}|_{\widetilde{\cF}_g^s})=0$. 
\end{lemma}
\begin{proof}
Write $Y=\widetilde{\cF}_g-\widetilde{\cF}^s_g=\cH^g_{A_{1,1}}\cup \cH^g_{A_2}$. We first compute $\rH^3(\widetilde{\cF}_g^s,\CC)$. Note that there is a topological commutative diagram of exact sequences
\begin{equation}
\xymatrix{
    \cdots \to 0= \rH^3(\widetilde{\cF}_g,\CC)\ar[r] \ar[d]^{\cap~ [\widetilde{\cF}_g]}& \rH^3(\widetilde{\cF}_g^s, \CC)\ar[r]\ar[d]^{\cap~[\widetilde{\cF}_g^s]} & \rH^4_{Y}(\widetilde{\cF}_g,\CC) \ar[r]^\rho\ar[d]^{\cap~[Y]} & \rH^4(\widetilde{\cF}_g,\CC)\ar[d]  \\ \cdots \to \rH_{35}(\widetilde{\cF}_g,\CC)\ar[r] &\rH_{35}(\widetilde{\cF}^s_g, \CC)\ar[r]& \rH_{34}(Y,\CC) \ar[r]& \rH_{34}(\widetilde{\cF}_g,\CC)
    } 
\end{equation}
Here, the first row is the classical sequence of cohomology (and with supports in $Y$) associated to the pair $(\widetilde{\cF}_g,Y)$, 
the second row is the standard exact sequence of Borel-Moore homology of the complement $\widetilde{\cF}_g^s=\widetilde{\cF}_g-Y$, the vertical arrows are given by cap products with the fundamental classes. 

Note that $\widetilde{\cF}_g$ and $\widetilde{\cF}_g^s$ are  finite quotient of smooth quasi-projective varieties,  the cap products $\cap ~[\widetilde{\cF}_g]$ and  $\cap ~[\widetilde{\cF}_g^s]$ are isomorphisms on all degrees.  Hence the map  $$\cap~[Y]:  \rH^4_{Y}(\widetilde{\cF}_g,\CC)\to \rH_{34}(Y,\CC) $$ is also an isomorphism. The space $\rH^4_Y(\widetilde{\cF}_g,\CC)$  is spanned by the fundamental classes of irreducible components of $Y$.  It follows from Proposition \ref{Linearindep} that $\rho$ is injective. This implies $\rH^3(\widetilde{\cF}_g^s,\CC)=0$.  

A similar argument shows that $\rH^1(\widetilde{\cF}_g, \underline{\Lambda_g}|_{\widetilde{\cF}_g^s})=0$.
\end{proof}

\subsection{Families of K3 surfaces in $\cF_g$}
Regard $\cF_g$ as the coarse moduli space of polarized K3 surfaces with at worst isolated ADE singularities. Let us construct  two-dimensional proper families of polarized K3 surfaces in $\cF_g$ for all $g$.  

\subsection*{Hyperelliptic family for $g>1$ odd} Let $\Sigma_1=\PP^2$. Let us take $\cC_1\subseteq \Sigma_1 \times \PP^1\times \PP^1$  a general member in  $| \cO_{\PP^2}(2)\boxtimes\cO_{\PP^1\times \PP^1}(-2K_{\PP^1\times \PP^1})|$.  Each fiber $\cC_{1,t}$ of the projection $\cC_1\to \Sigma_1$  over a point $t\in \Sigma_1$ is a curve in $|\cO_{\PP^1\times \PP^1}(-2K_{\PP^1\times \PP^1})|$ with at worst a cusp  or two nodes.  Consider the double covering 
$$\cX_1\to \Sigma_1 \times \PP^1\times \PP^1$$
branched over $\cC_1$, equipped with natural projections to the three factors. The first projection gives a family $\phi_1:\cX_1\to \Sigma_1$ of surfaces over $\Sigma_1$.  For any $t\in\Sigma_1$, the fiber $\cX_{1,t}=\phi_1^{-1}(t)$  is a double covering of $\PP^1\times \PP^1$ branched over the curve $\cC_{1,t}$. It is a hyperelliptic K3 surface and it is smooth (resp.~with $A_1$ or $A_2$ singularity) if and only if the branching curve $\cC_{1,t}$ is smooth (resp.~nodal or cuspidal).  Moreover, the family  $\cX_1\to \Sigma_1$ has  a relative polarization   $$\cL_1=\phi_1^\ast\cO_{\PP^2}(1)+\pi_2^\ast \cO_{\PP^1}(1)+ \frac{g-1}{2}\pi_3^\ast \cO_{\PP^1}(1),$$
of odd genus  $g$, which induces  a map $\Sigma_1\to \cF_g$ factoring through the natural map $\Sigma_1\to |-2K_{\PP^1\times \PP^1}|$.  Conversely, a hyperelliptic K3 surface in $\cF_g$ ($g>1$ is odd) with at worst ADE singularities arises as a double covering of $\PP^1\times \PP^1$ branching over an integral curve  in $|-2K_{\PP^1\times \PP^1}|$.

\subsection*{Hyperelliptic family for $g>2$ even} Let $\Sigma_2=\PP^2$. Consider the Hirzebruch surface $\bF_1=\PP(\cO_{\PP^1}\oplus \cO_{\PP^1}(-1))$, we denote by $s$ the section of $\bF_1\to \PP^1$ and $h$ the fiber class.  Take 
$\cC_2\subseteq \PP^2\times \bF_1$ a general member
in $|\cO_{\PP^2}(2)\boxtimes\cO_{\FF_1}(-2K_{\bF_1})|$ and let $\cX_2\to \Sigma_2 \times \bF_1$ be the double covering branched over $\cC_2$.   The projection $$\phi_2:\cX_2\to \Sigma_2$$ to the first factor gives a family of hyperelliptic K3 surfaces. The fiber $\cX_{2,t}=\phi_2^{-1}(t)$ is a double covering of $\bF_1$ branched over the curve $\cC_{2,t}$. Similarly as before, this is a family of hyperelliptic K3 surfaces which carries a relative polarization 
$$\cL_2=(\pi_1')^\ast \cO_{\PP^2}(1)+ (\pi_2')^\ast \cO_{\bF_1}(s+\frac{g}{2}h) $$
of even genus  $g$. Similarly, we have a moduli map $\Sigma_2\to \cF_g$ factoring through $\Sigma_2\to |-2K_{\bF_1}|$ for a generic choice of $\cC_2$.

\subsection*{Unigonal family}  The construction of the family of unigonal K3 surfaces needs more attentions. 
To get a two-dimensional proper family of unigonal K3 surfaces, we follow the construction in \cite[\S 4.2.2]{LaO19} to proceed as follows:

Let $\Sigma_3=(\PP^2)^*$ be the dual projective plane, $\Phi=\{([L], p)\in (\PP^2)^*\times \PP^2|~p\in L\}$ the universal line, and $\pi_1: \Phi \to (\PP^2)^*$ and $\pi_2: \Phi \to \PP^2$ the projections.

Consider the fourfold
$$\cY=\PP_{\Phi}(\cO_{\Phi}\oplus \cO_{\Phi}(-4)),$$
where $\cO_{\Phi}(1)$ is the pull-back of $\cO_{\PP^2}(1)$ along $\pi_2: \Phi\to \PP^2$. Let $\cA=\PP_{\Phi}(\cO_{\Phi}(-4))$ and $\cF$ the pull-back of a line in $\PP^2$ along $\cY\to \Phi \xrightarrow{\pi_2} \PP^2$. Take a general member
$$\cB\in |3\cA+12\cF|$$
such that it is smooth and disjoint from $\cA$, and form the double covering $\cX_3\to \cY$ branched over $(\cA+\cB)$. Then the composite $$\cX_3\to \cY \to \Phi \xrightarrow{\pi_1} (\PP^2)^*$$ gives a family of unigonal K3 surfaces, denoted by $\phi_3: \cX_3\to \Sigma_3$.

\subsection{The NL-number computation}
For the three families $\phi_i$ ($i=1, 2, 3$) constructed above, we are going to compute the intersection numbers
\begin{equation}\label{NLnum}
    \Sigma_i\cdot \cH_{A_{1,1}}^g~\hbox{and}~\Sigma_i\cdot \cH_{A_2}^g
\end{equation} 
in $\cF_g$.  

\subsection*{Hyperelliptic family} 
We first consider the hyperelliptic families. By taking $\cC_1$ and $\cC_2$ generically, we can further assume that all the fibers of $\cC_i\to \Sigma_i$ have only trivial automorphisms. Then we have embeddings
$$ \psi_1:\Sigma_1\to |-2K_{\PP^1\times \PP^1}|~\hbox{and}~ \psi_2:\Sigma_2\to |-2K_{\bF_1}|.$$

Let us first consider the case $i=1$, and similar arguments will apply to the case $i=2$. Denote by 
\begin{center}
$\cA_{1,1}\subseteq |-2K_{\PP^1\times \PP^1}|$ and $\cA_2\subseteq |-2K_{\PP^1\times \PP^1}|$  
\end{center}
the Zariski closure of the  binodal locus and respectively the cuspidal locus in $| -2K_{\PP^1\times \PP^1}|\cong \PP^{24}$ which parametrize curves in $|-2K_{\PP^1\times \PP^1}|$ with at worst two nodes (resp.~a cusp).   It is known that $\cA_{1,1}$ and $\cA_{2}$ are both of codimension two. Then we have 
\begin{equation}
 \begin{aligned}
        \Sigma_1\cdot \cH_{A_{1,1}}^g &=\Sigma_1\cdot \cA_{1,1}\\
         \Sigma_1\cdot  \cH_{A_{2}}^g &=\Sigma_1\cdot \cA_{2}.
 \end{aligned}
\end{equation}
where the right hand side is taking intersections in  $|-2K_{\PP^1\times \PP^1}|$, since by construction the intersection numbers are the numbers of binodal and cuspidal members in the two-dimensional family of the branching curves  $\cC_1\to \Sigma_1$ respectively.  It suffices to compute the degree of $\cA_{1,1}$ and $\cA_{2}$ in $|-2K_{\PP^1\times \PP^1}|$. This becomes a classical enumeration problem for singular members in a family of curves, which has been studied extensively in \cite{Kazarian}, \cite{Kazarian2}, etc.  

By taking a general two-dimensional linear subspace  $B=\PP^2\subseteq |-2K_{\PP^1\times \PP^1} |$, we obtain a family $f:\cC\to B$ of curves in $|-2K_{\PP^1\times \PP^1}|$ called a {\it net of curves} in $|-2K_{\PP^1\times \PP^1}|$. Let $\Delta_f \subseteq B$ be the discriminant curve of $f$, i.e. the locus in $B$ over which the fibers are singular. Suppose that $f$ has at worst cuspidal or binodal fibers,  and let $a_2(f)$ and $a_{1,1}(f)$ be the numbers of cuspidal and binodal fibers of $f$ respectively. In this case, we have
\begin{center}
$\deg(\cA_2)=a_2(f)$ and $\deg(\cA_{1,1})=a_{1,1}(f)$.
\end{center}
By our assumption, the discriminant curve $\Delta_f$ has only nodes and cusps as its singularities and $a_2(f)$ is equal to the number of cusps of $\Delta_f$ while $a_{1,1}(f)$ is equal to the number of nodes of $\Delta_f$. Then we have the following explicit formulae which compute $a_{1,1}(f)$ and $a_2(f)$. 
 
\begin{proposition}\label{netenu}({\it cf}.~\cite[\S 11.4.4]{EH2016})
Let $S$ be a smooth projective surface and $L$ a very ample line bundle on $S$. Let $f:\cC\to \PP^2$ be a general net of curves in $|L|$ whose fibers are at worst cuspidal or binodal. Then the numbers of cuspidal fibers and binodal fibers of $f$ are given by 
\begin{align}
    a_2(f)&=2g-d+2(e-1)\\
    a_{1,1}(f)&=d-3g-e(e-1)+\frac{3}{2}(e-1)(e-2)
\end{align}
respectively, where $c_i=c_i(\Omega_S^1)$,  $\alpha=c_1(\cL)$, $g=\frac{1}{2}(9\alpha^2+9\alpha c_1+2c_1^2)+1$, $d=3\alpha^2+\alpha c_1$ and  $e=3\alpha^2+2\alpha c_1+c_2$.
\end{proposition}

Note that for $(S,L)= (\PP^1\times \PP^1, |-2K_{\PP^1\times \PP^1}|)$ or $(\bF_1, |-2K_{\bF_1}|)$, we always have 
$$\alpha^2=32,~c_1^2=8,~\alpha c_1=-16,~c_2=4.$$
By applying Proposition \ref{netenu} to a general net of curves in $|-2K_{\PP^1\times \PP^1}| $ or $|-2K_{\bF_1}|$,  one can compute that in both cases the numbers of cuspidal and binodal members in such a net are
\begin{align}
    a_2(f)&=216\\
    a_{1,1}(f)&=1914,
\end{align}
respectively. Since the surface $\Sigma_i$ has degree $4$ in the linear system, it follows that 
\begin{proposition}
For $i=1,2$,   $\Sigma_i\cdot \cH^g_{A_{2}}=864$ and $\Sigma_i\cdot \cH^g_{A_{1,1}}=7656$.
\end{proposition}

\subsection*{Unigonal family}
Now we consider the unigonal family $\phi_3:\cX_3\to \Sigma_3=(\PP^2)^*$. Let $$\psi_3:\Sigma_3\to \cF_{g}$$ be the induced moduli map.  As above, the intersection numbers $\deg(\psi_3^\ast (\cH^g_{A_2}))$ and $\deg(\psi_3^\ast (\cH^g_{A_2}))$  are, by construction, the numbers of binodal and cuspidal members in the two-dimensional family of the branching curves  $\cB\to \Sigma_3$.

Let $M$ be the locus in $\cB$ of the singularities of all singular fibers of $\cB\to \Sigma_3$, which is smooth by genericity of $\cB$.  They fit into the following commutative diagram:
\[
\begin{tikzcd}
M \ar[r, hook, "i"] \ar[drr, "f"']&\cB \ar[r, hook] \ar[dr, "h"]&\cY \ar[d, "\pi"]\\
 & & \Sigma_3
\end{tikzcd}.
\]
Note that the family $\pi$ in which the family $\cB\to \Sigma_3$ of curves is embedded is no longer a trivial family of surfaces.

The image $\Delta_h=f(M)\subseteq \Sigma_3$ is exactly the discriminant curve of $h$. By genericity of $\cB$, we may assume that the singular fibers of $h:\cB\to \Sigma_3$ are at worst cuspidal or binodal. As before,  the intersection number $\psi_3^\ast(\cH^g_{A_2})=a_2(h)$ is equal to the number of cusps of $\Delta_h$ while  $\psi_3^\ast(\cH^{g}_{A_{1,1}})= a_{1,1}(h)$ is equal to the number of nodes of $\Delta_h$. They satisfy the natural relation
\begin{equation}\label{bi-cusp relation}
    2(a_2(h)+a_{1,1}(h))=2(p_a(\Delta_h)-p_g(\Delta_h))=\deg \DD(f)
\end{equation}
where $\DD(f)\in \CH^1(\cM)$ is the double point class of $f$, given by (\cite[Theorem 9.3]{Fulton98})
\begin{equation}\label{D-formula}
    \DD(f)=f^*f_*[M]-(c_1(f^*\cT_{\Sigma_3})-c_1(\cT_{M}))\cap [M].
\end{equation}
According to \cite[\S 8-10]{Kazarian}, we have the following formula for $a_2(h)$
\begin{equation}\label{cusp-formula}
    a_2(h)=\deg ((c_1(\Omega_{\pi}\otimes\cO_{\cY}(\cB))-c_1(\cT_{\pi}))\cap [M]).
\end{equation}
This gives 
\begin{proposition}\label{NL-unigonal}
The intersection numbers are 
\begin{align*}
    a_2(h)&=816\\
    a_{1,1}(h)&=33480.
\end{align*}
\end{proposition}
\begin{proof}
We need to compute the class of $M$. Let $dh:\cT_{\cB}\to h^*\cT_{\Sigma_3}$ be the differential of $h$. Consider the projective bundle $p:\cP=\PP_{\cB}(h^*\cT_{\Sigma_3})\to \cB$. The composition
$$ p^*\cT_{\cX_3}\xrightarrow{p^*dh} p^*h^*\cT_{\Sigma_3}\to \cO_{\cP}(1)$$
gives a section $s\in \rH^0(\cP,\cO_{\cP}(1)\otimes p^*\Omega_{\cB})$. Let $\widetilde{M}$ be the zero locus of $s$. According to \cite[Lemma 2.3.1]{Kazarian2}, the restriction of $p$ to $\widetilde{M}$ is an isomorphism  onto  $M$.  In particular we have
\begin{equation}\label{M-class}
    \begin{split}
        i_*[M]&=p_*c_{\rm top}(\cO_{\cP}(1)\otimes p^*\Omega_{\cB})\\
        \deg c_1(\cT_M)&=\deg c_1(\cT_{\widetilde{M}})
    \end{split}
\end{equation}

Set
\begin{align*}
    \alpha_i&=c_i(\Omega_{\cX_3}),~
    \beta_j=c_j(\cT_{\Sigma_3}),~
    \delta=h_*(\alpha_1)\beta_1+h_*(\alpha_2)
\end{align*}
for $i=1,2,3$ and $j=1,2$.  Combining \eqref{bi-cusp relation}, \eqref{D-formula}, \eqref{cusp-formula} and \eqref{M-class},  we get 
\begin{equation}\label{formulae_for_A2_and_A11}
    \begin{aligned}
    a_2(h)=&~4h_*(\alpha_1)\beta_1^2+2h_*(\alpha_1^2+\alpha_2)\beta_1-2h_*(\alpha_1)\beta_2+2h_*(\alpha_1\alpha_2)\\
    2(a_2(h)+a_{1,1}(h))=&~\delta^2-\beta_1\delta+(-2h_*(\alpha_1\alpha_2)-h_*(\alpha_3) -2h_*(\alpha_1^2+\alpha_2)\beta_1\\& +h_*(\alpha_1)(-4\beta_1^2+3\beta_2))
    \end{aligned}
\end{equation}
Set $\zeta=c_1(\cO_{\Sigma_3}(1))$. Then the classes of $h_\ast(\alpha_i)$ and $h_\ast(\alpha_1\alpha_2)$ and $\delta$ are given in the following table:
\begin{center}
    \begin{tabular}{c|c|c|c|c|c}
    $h_*(\alpha_1)$ & $h_*(\alpha_2)$ & $h_*(\alpha_3)$ & $h_*(\alpha_1^2)$ & $h_*(\alpha_1\alpha_2)$ & $\delta$\\
    \hline
    $18$ & $210\zeta$ & $-450\zeta^2$ & $36\zeta$ & $-600\zeta^2$ & $264\zeta$\\
    \end{tabular}
\end{center}
One can get the numbers  by substituting these data into (\ref{formulae_for_A2_and_A11}).
\end{proof}

\begin{remark}\label{delta_is_discriminant}
In the above computation, $\delta$ is nothing but the class of the discriminant curve $\Delta_h$ of $h$.
\end{remark}

Now we are ready to derive the following 
\begin{proposition}\label{Linearindep}
Let $\widetilde{\cH}_{A_{1,1}}$ and $\widetilde{\cH}_{A_{2}}$ be the preimages of $\cH_{A_{1,1}}$ and $\cH_{A_2}$ in $\widetilde{\cF}_g$ respectively. Then the fundamental classes of irreducible components of $\widetilde{\cH}_{A_{1,1}}$ and $\widetilde{\cH}_{A_{2}}$ are linearly independent in $\rH^4(\widetilde{\cF}_g, \QQ)$. 
\end{proposition}
\begin{proof}
As the covering map $\widetilde{\cF}_g\to \cF_g$ is branched over $\cH^g_{0,-2}$ containing $\cH^g_{A_{1, 1}}\cup\cH^g_{A_2}$ , we know that the preimages of the irreducible components of $\cH^g_{A_{1, 1}}$ or $\cH^g_{A_2}$ remain irreducible. It suffices to show that the fundamental classes of irreducible components of $\cH^g_{A_{1, 1}}$ and $\cH^g_{A_2}$ are linearly independent in $\rH^4(\cF_g, \QQ)$. According to Proposition \ref{irred-comp}, the irreducible components of $\cH^g_{A_{1,1}}$ and $\cH^g_{A_{2}}$ are described as follows: 
\begin{itemize}
    \item [i)] if $g\equiv 0 ~\hbox{or} ~1\mod 4$, $\cH^g_{A_{1,1}}$ and $\cH^g_{A_{2}}$ are irreducible. 
    \item [ii)] if $g\equiv 2\mod 4$, $\cH^g_{A_{2}}$ is irreducible and $\cH^g_{A_{1,1}}=\cH_{A_{1,1}}^{g'}\cup \cH_{A_{1,1}}^{g''}$. 
    \item [iii)] if $g\equiv 3\mod 4$, $\cH^g_{A_{2}}$ is irreducible and $\cH^g_{A_{1,1}}=\cH_{A_{1,1}}^{g'}\cup \cH_{A_{1,1}}^{g'''}$. 
\end{itemize}
   
In case i), we take the testing families $\phi_i:\cX_i\to \Sigma_i, i=1,2,3$ depending on the parity of $g$. We have the following intersection matrix
\begin{center}
    \begin{tabular}{c|cc}
      & $\cH_{A_{1,1}} $ &  $\cH_{A_2}$ \\\hline
      $\Sigma_i$  &  7656  & 864 \\
      $\Sigma_3$  &  33480 & 816
    \end{tabular} 
\end{center} for $i=1$ or $2$. As the matrix is non-degenerate, we can conclude our assertion easily. 

In case ii), one can take a general two-dimensional proper family $\phi_4:\cX_4\to \Sigma_4$  of K3 surfaces in $\scF_g$ such that there exist finitely many fibers of $\phi_4$ lying in $\cH^{g''}_{A_{1,1}}$. This is valid because the boundary of  the Baily-Borel compactifcation of $\cF_g$ is only $1$-dimensional.  Next, note that $\Sigma_1$ and $\Sigma_3$ do not intersect with the NL-locus $\cH^g_{g-1,\frac{g-2}{2}}$ which contains $\cH^{g''}_{A_{1,1}}$.  The corresponding intersection matrix  of $\Sigma_i$ with  $\cH^g_{A_{1,1}}$ and $\cH^g_{A_2}$ is given by 
\begin{center}
    \begin{tabular}{c|ccc}
      & $\cH^{g'}_{A_{1,1}} $ & $\cH_{A_{1,1}}^{g''} $ &  $\cH^g_{A_2}$  \\\hline
      $\Sigma_1$  &  7656  & 0 & 864 \\
       $\Sigma_3$  &  33480 & 0 & 816 \\
       $\Sigma_4$  &  $a$ & $b\neq 0$ & $c$
    \end{tabular} 
\end{center} 
which is clearly non-degenerate. Therefore, the fundamental classes of $\cH_{A_{1,1}}^{g'}, \cH_{A_{1,1}}^{g''}$ and $\cH^g_{A_2}$ are linearly independent. 

In case iii), one can obtain a similar intersection matrix. The assertion follows similarly.  
\end{proof}

\subsection{NL-numbers via Kudla's modularity conjecture}
We will provide an alternative way to compute the binodal and cuspidal fibers on a family of unigonal K3 surfaces which  uses less classical enumerative geometry.  For any two-dimensional nontrivial  family of unigonal K3 surfaces (with only ADE singularities)
$$f:\cX\to \Sigma $$
over a smooth projective surface $\Sigma$, it induces a map $\psi:\Sigma\to \cF_U\cong \Gamma_U \backslash \rD_{U^\perp}$. To compute the Noether-Lefschetz numbers with respect to $f$, we can applying Theorem \ref{Modularity} (see also Example \ref{Modularity-unimodular} and Example \ref{Elliptic}) to $\cF_U$ and  get that  the generating series 
\begin{equation}
\begin{aligned}
      \Theta_\psi(\tau): &=\sum\limits_{\beta\geq 0} (\deg(\psi^\ast Z(\beta))) e^{2\pi i \mathrm{tr} (\beta \tau )}  \\ &=\sum\limits_{k,l,m\geq 0} N_{k,l,m}\tilde{q}^kp^lq^m
\end{aligned}
\end{equation}
is a classical scalar-valued Siegel modular form of weight $10$ and genus $2$, where we write $\tau=\left(\begin{array}{ccc} z_1  & z_2 \\ z_2 & z_3 \end{array}\right)\in \HH_2$,  $\tilde{q}=e^{2\pi i z_1}$, $p=e^{2\pi i z_2}$ and $q=e^{2\pi i z_3}$ for $z_1,z_3\in \HH, z_2\in \CC$.  Then we see that all the Noether-Lefschetz numbers can be read off from the Fourier coefficients of some Siegel modular form. The geometric meaning of  some coefficients $N_{k,l,m} $ are given as below (see also Proposition \ref{lattice-cuspidal-binodal}):
\begin{itemize}
    \item the constant term $N_{0,0,0}=(R^2f_\ast\cO_\cX)^2 $ is the  self-intersection of the Hodge line bundle on $\Sigma$. 
    \item $N_{0,0,1}=\deg((\psi^\ast \lambda )|_\Delta)$ is the degree of the Hodge line bundle on the discriminant curve  $\Delta\subseteq \Sigma$.   
    \item $N_{1,1,1}=\deg(\psi^\ast Z(\left(\begin{array}{ccc} -2  & 1 \\ 1 & -2 \end{array}\right) ))$ is twice the  number of cuspidal fibers. 
    \item $N_{1,0,1}=\deg(\psi^\ast(Z(\left(\begin{array}{ccc} -2  & 0 \\ 0 & -2 \end{array}\right) )))$ is  twice the number of binodal fibers. 
\end{itemize}

Next, due to Igusa's work, a Siegel modular form of weight $10$ must be a linear combination of $E_4^{(2)}E_6^{(2)}$ and $\chi_{10}$. Here,  $E_{n}^{(2)}$ is the Siegel Eisenstein series  of weight $n$ and genus $2$ defined by
\begin{align*}
    E_n^{(2)}(\tau)=\frac{1}{2\zeta(n)} \sum\limits_{(c,d)} \det(c\tau+d)^{-n}, \tau\in \HH_2
\end{align*}
where summation runs over the bottom row of a complete system of representatives $(c,d)$
of elements in $\{ \left(\begin{array}{ccc} \ast  & \ast  \\ 0 & \ast \end{array}\right) \in \Sp_4(\ZZ)\}\backslash \Sp_{4}(\ZZ)$; $\chi_{10}$ is the Igusa cusp form defined by the infinite product 
\begin{equation}   
\begin{aligned}
\chi_{10}(\tau)& =\tilde{q} p q\prod\limits_{(r,s,t)> 0} (1-\tilde{q}^rp^s q^t)^{c(4rt-s^2)}  \\
&=\tilde{q}pq-2\tilde{q}q-16 \tilde{q}pq^2+\cdots 
\end{aligned}
\end{equation} 
where $(r,s,t)> 0$ means either $r>0, t\geq 0$ or $r\geq 0,t>0$ or $r=t=0, s<0$;  $c(m)$ is the Fourier coefficient of a modular form of weight $-\frac{1}{2}$ given by 
\begin{equation}\label{wt-1/2}
    \sum\limits_{m\geq 0} c(m)q^m=2q^{-1}+20-128q^3+216q^4-1026q^7+1618q^8+\cdots.
\end{equation}
See also \cite[\S 0.2]{OP18} for the explicit construction of \eqref{wt-1/2}.  This provides  a way to construct many testing surfaces which  implies   Proposition \ref{Linearindep}.

\begin{proposition}
Suppose there exists a nontrivial family of Weierstrass K3 surfaces with only trivial automorphisms (i.e. the identity and the involution) over a smooth projective surface  $\Sigma$ which induces a morphism  $\psi: \Sigma\rightarrow \cF_U$.  If the self-intersection of the pullback $\psi^\ast$ of the Hodge line bundle on $\Sigma$ is not divisible by $17$ or $23$, then $\widetilde{\cH}_{A_{1,1}}$ and $\widetilde{\cH}_{A_2}$ are linearly independent in $\rH^4(\widetilde{\cF}_g, \QQ)$. As a consequence, one can obtain Proposition \ref{Linearindep}. 
\end{proposition}
\begin{proof}
Note that 
\begin{equation}
\begin{aligned}
      E_4^{(2)}E_6^{(2)}=~&(1+240\tilde{q}+240q+\cdots+13440\tilde{q}pq+ 30240\tilde{q}q+\cdots) \times \\&(1-504\tilde{q}-504q +\cdots+44352\tilde{q}p q+166320\tilde{q}q+\cdots)\\
      =~&1-264\tilde{q}-264q +\cdots+57792 \tilde{q}pq-45360 \tilde{q}q+\cdots . 
\end{aligned} 
\end{equation}
Kudla's modularity theorem shows that $\Theta_\psi$ must have the form
$$\Theta_\psi=aE_4^{(2)}E_6^{(2)}+b\chi_{10}$$
for some $a,b\in \QQ$. As the coefficients  $N_{0,0,0}$ and $N_{1,1,1}$ are integers, we know that $a$ and  $b$ are integers as well.  We can write 
$$\deg (\psi^{\ast} Z(\left(\begin{array}{ccc} -2  & 1 \\ 1 & -2 \end{array}\right)))= 57792a+b;~\deg \psi^\ast (Z(\left(\begin{array}{ccc} -2  & 0 \\ 0 & -2 \end{array}\right)))=-45360a -2b.$$
The assumption then implies that $\frac{57792a+b}{-45360a -2b}\neq \frac{864}{7656}$, which proves Proposition \ref{Linearindep}.

For the last statement, we need to construct such a family. The idea is to make use of Miranda's construction.  As in \S \ref{Elliptic}, we set $W=\PP(2^{(9)}, 3^{(13)})$ for convenience and one can regard  $\cF_U$ as an open subset of $W\q \SL_2$ via the isomorphism \eqref{OO-iso}.  Embed $W$ into $|\cO_{W}(6\cdot 2^n)|$  with  $n$ sufficiently large.   One can take general hyperplanes $H_1,\ldots,H_{16}$ in $|\cO_{W}(6\cdot 2^n)|$, such that they meet transversally with $W$ and their intersection  $\Sigma=W\cap H_1\cap\ldots \cap H_{16}$ is contained in the  open locus $W^\circ$ consisting of Weierstrass K3 surfaces with only trivial automorphisms.  This is valid because the complement of $W^\circ$ in $W$ has codimension $>2$.  Let $\psi:\Sigma\to \cF_U$ be the map induced by the quotient map $W^\circ\to W^\circ/ \SL_2\subseteq \cF_U$.  According to the computation in \cite[Proposition 4.1.1]{LaO19}, $\cO_{W^\circ}(1)$ descends to the Hodge line bundle $\lambda$ on $W^\circ/\SL_2$.  One can compute that  the degree of $(\psi^\ast \lambda)^2$ equal to the degree of $S$ in $|\cO_{W}(6\cdot 2^n)|$, which is only divisible by $2$ and $3$.  Thus we can conclude our assertion.

\end{proof}

As an example,  one can verify our computations for NL-numbers on $\psi_3$. 
\begin{corollary}
\begin{equation}
   \Theta_{\psi_3}=E_4^{(2)}E_6^{(2)}-56160\chi_{10}.
 \end{equation}
\end{corollary}

\begin{proof}
From the computation in Proposition \ref{NL-unigonal}, some  coefficients of $\Theta_{\psi_3}$ are as follows:
\begin{itemize}
    \item $N_{0,0,1}=264$ (see Remark \ref{delta_is_discriminant}),
    \item $N_{1,1,1}=2\times a_2(h)=1632$,
    \item $N_{1,0,1}=2\times a_{1,1}(h)=66960$.
\end{itemize}
The assertion follows directly. 
\end{proof}

\subsection{Proof of the main theorem} 
Using the vanishing result, we can immediately get

\begin{lemma}\label{inj}
 	The natural map
 	\begin{equation}
 	\rH^{4}_{\DB}(\widetilde{\scS}^s_g,\QQ(2))\rightarrow \rH^{4}(\widetilde{\scS}^s_g,\QQ)
 	\end{equation}
 	is injective for $r<3$. Moreover, the map $\rH^{4}_{\DB}(\widetilde{\scS}^{sm}_g,\QQ(2))\rightarrow \rH^{4}(\widetilde{\scS}^{sm}_g,\QQ)$ is injective as well. 
 \end{lemma} 
 
\begin{proof}
Consider the smooth and proper map $\widetilde{\pi}: \widetilde{\scS}_g^s\to \widetilde{\scF}_g^s$, we can apply Deligne's decomposition theorem  \cite{De68} to get
	\begin{equation}\label{Dec}
	\rH^k(\widetilde{\scS}^s_g,\CC)\cong \bigoplus\limits_{i+j=k} \rH^i(\widetilde{\scF}^s_g; R^{j}\widetilde{\pi}_\ast (\CC)).
	\end{equation}	
According to Proposition \ref{vanishing}, we have  
$$\rH^{3}(\widetilde{\scS}_{g}^s,\QQ)=0.$$
The assertion follows from the long exact sequence \eqref{long_exact_seq_for_stacks}. The last assertion holds because $\widetilde{\scS}^{sm}_g$ has codimension $>2$ in $\widetilde{\scS}^{s}_g$.
\end{proof}

For each $\alpha\in \CH^2(\widetilde{\scS}^s_g)$,  according to \cite{BL17}, there exists $m\in\QQ$ such that 
\begin{equation}\label{fran}
\cl_{\rm B}(\alpha-mc_{2} (\cT_{\widetilde{\pi}}))\in \rH^4(\widetilde{\scS}^s_g, \QQ)
\end{equation}
can be written as a linear combination of universal line bundles over some moduli spaces of lattice polarized oriented K3 surfaces in $\rH^4(\widetilde{\scS}^s_g,\QQ)$. Then  Lemma \ref{inj} implies that this also holds in the Deligne-Beilinson cohomology group. Thus we obtain

\begin{theorem}\label{GFC_for_oriented}
Let $\cT_{\widetilde{\pi}}$ be the relative tangent bundle of $\widetilde{\pi}:\widetilde{\scS_g^s}\to \widetilde{\scF_g^s}$. Then for any $\widetilde{\alpha} \in \CH^2(\widetilde{\scS_{g}^s)}$, there exists a rational number $m\in \QQ$ such that the class 
\begin{equation}
    \cl_{\DB}(\widetilde{\alpha}-m c_2(\cT_{\widetilde{\pi}})) \in \rH^4_{\DB}(\widetilde{\scS_g^s}, \QQ(2))
\end{equation}
is supported on a proper closed subset of $\widetilde{\scF_{g}^s}$, i.e.  the restriction of $ \cl_{\DB}(\widetilde{\alpha}-m c_2(\cT_{\widetilde{\pi}})) $ is zero on $\widetilde{\pi}^{-1}(W)$ for some open substack $W$ in $\widetilde{\scF}_g^s$. 
\end{theorem}

It remains to prove Theorem \ref{mainthm}.  We can assume $g>2$ so that $\scS_g^\circ\to \scF_g^\circ$ is a smooth morphism between smooth varieties. Recall that we have the following commutative diagram
\begin{equation}\label{Cart-dia}
    \xymatrix{
    \widetilde{\scS_{g}^{s}} \ar[d]^{\widetilde{\pi}}\ar[r] \ar[d] & \scS_{g} \ar[d]& \scS_{g}^{\circ} \ar@{_{(}->}[l] \ar[d]^{\pi}\\
    \widetilde{\scF_{g}^{s}} \ar@{^{(}->}[r]  \ar[r] & \scF_{g} & \scF_{g}^{\circ} \ar@{_{(}->}[l]
    }
\end{equation}
We let $\widetilde{\pi^\circ}:\widetilde{\scS_{g}^\circ}\to \widetilde{\scF}_g^\circ $ be the pullback of $\pi$ via \eqref{Cart-dia} and set $v:\widetilde{\scS_{g}^\circ}\to \scS_{g}^\circ$ to be the covering map. 
According to Theorem \ref{GFC_for_oriented},  for any class $\alpha\in \CH^2(\scS_{g}^\circ)$, there exists a rational number $m\in \QQ$ such that $\cl_{\DB}(v^\ast (\alpha)-m c_2(\cT_{\widetilde{\pi^{\circ}}}))$ is supported on a proper closed subset of $\widetilde{\scF_{g}^\circ}$.  

Applying  the push-forward map $$v_{!}: \rH^4_{\DB}(\widetilde{\scS_{g}^\circ}, \QQ(2)) \to \rH^4_{\DB}(\scS_{g}^\circ, \QQ(2))$$  to $\cl_{\DB}(v^\ast (\alpha)-m c_2(\cT_{\widetilde{\pi^{\circ}}}))$ (such a push-forward exists for any proper morphism between complex algebraic manifolds, since there exists also the theory of  Deligne-Beilinson homology, which together with Deligne-Beilinson cohomology form a twisted Poincar\'e duality theory, cf. \cite{Jannsen1988}), we find that $\cl_{\DB}(\alpha-m c_2(\cT_{\pi}))$ is likewise supported on a proper closed subset of $\scF_{g}^\circ$.  This finishes the proof of Theorem \ref{mainthm}.

\bibliographystyle {plain}
\bibliography{main}

\end{document}